\documentclass[12pt]{amsart}
\usepackage{amsmath}
\usepackage{graphicx}
\usepackage{amssymb,amsthm}
\usepackage{hyperref}
\usepackage{esint}
\usepackage{epstopdf}
\usepackage{color}
\usepackage{url}
\usepackage[top=1in, left=1in, right=1in, bottom=1in]{geometry}

\newtheorem{theorem}{Theorem}[section]
\newtheorem{lemma}[theorem]{Lemma}
\newtheorem{proposition}[theorem]{Proposition}
\newtheorem{corollary}[theorem]{Corollary}

\newtheorem{remark}[theorem]{Remark}
\newtheorem{definition}[theorem]{Definition}

\def\R{{\mathbb R}}

\def\cE{{\mathcal F}}

\def\cH{{\mathcal H}}

\def\cK{{\mathcal K}}

\def\cQ{{\mathcal Q}}

\def\cU{{\mathcal U}}

\def\a{\alpha}
\def\b{\beta}
\def\e{\varepsilon}
\def\d{\delta}

\def\k{\kappa}
\def\l{\lambda}
\def\L{\Lambda}
\def\m{\mu}
\def\n{\nabla}
\def\p{\partial}
\def\r{\rho}
\def\s{\sigma}
\def\t{\tau}

\def\w{\omega}
\def\W{\Omega}
\def\g{\gamma}
\def\z{\zeta}

\def\1{\left(}
\def\2{\right)}
\def\3{\left\{}
\def\4{\right\}}
\def\8{\infty}
\def\sm{\setminus}
\def\ss{\subseteq}
\def\cc{\subset\subset}

\DeclareMathOperator*{\diam}{diam}
\DeclareMathOperator*{\supp}{supp}

\usepackage{enumerate}

\def\Wb{\Omega_*}

\def\Wc{\bar{\Omega}_*}

\begin{document}

\title{Regularity for Shape Optimizers: The Degenerate Case}

\author{Dennis Kriventsov}
\address[Dennis Kriventsov]{Courant Institute of Mathematical Sciences, New York University, New York}
\email{dennisk@cims.nyu.edu}  

\author{Fanghua Lin}
\address[Fanghua Lin]{Courant Institute of Mathematical Sciences, New York University, New York}
\email{linf@cims.nyu.edu}

\date{October 27, 2017}

\begin{abstract}
	We consider minimizers of
	\[
	F(\l_1(\W),\ldots,\l_N(\W)) + |\W|,
	\]
	where $F$ is a function nondecreasing in each parameter, and $\l_k(\W)$ is the $k$-th Dirichlet eigenvalue of $\W$. This includes, in particular, functions
 $F$ which depend on just some of the first $N$ eigenvalues, such as the often studied $F=\l_N$. The existence of a minimizer, which is also a bounded set of 
finite perimeter, was shown recently. Here we show that the reduced boundary of the minimizers $\W$ is made 
up of smooth graphs, and examine the difficulties in classifying the singular points. Our approach is based on an approximation ("vanishing viscosity") argument, which--counterintuitively--allows us to recover an Euler-Lagrange equation for the minimizers which is not otherwise available.
\end{abstract}

\maketitle

\section{Introduction}\label{sec:intro}

Let $\l_k(\W)$ be the $k$-th eigenvalue of the Laplacian, with homogeneous Dirichlet boundary conditions, of a set $\W$. A commonly studied optimization problem involves searching for sets $\Wb$ of volume $1$ which have
\begin{equation}\label{eq:mainprob}
\l_k(\Wb) = m_k:= \inf \{\l_k(\W) : |\W|=1 \}.
\end{equation}
Understanding the $\Wb$ satisfying \eqref{eq:mainprob} for each $k$ is of great interest from the perspective of spectral theory of domains. It is relevant to Polya's conjecture that
\[
m_k \geq c(n) k^{\frac{2}{n}}
\]
(here $c(n)$ is a specific dimensional constant), to various generalizations of it, and generally to understanding quantitative forms of Weyl's asymptotic law and the constraints placed on a domain's shape by its spectrum. See \cite{lin} and references therein for a discussion of these topics. Optimization problems of this type, and various generalizations modeled on them, are also of importance in physics and in various applied contexts.

Despite a century of effort, sets $\Wb$ satisfying \eqref{eq:mainprob} are surprisingly poorly understood. The only cases resolved in a truly satisfactory manner are when $k=1$ (this is the well-known Faber-Krahn inequality, which asserts that $\Wb$ must be a ball) and when $k=2$ (this is a result attributed to Krahn and Szeg\H{o}, which guarantees that $\Wb$ is the union of two disjoint balls of volume $\frac{1}{2}$). When $k=3$, it was established by Wolf and Keller that $\Wb$ is connected, at least for $n=2$ and $3$ \cite{WK}. Beyond that, explicit results are limited, although there have been several interesting numerical studies recently \cite{O,AP,AO}. We refer to the book \cite{H} for useful discussion, many interesting examples and open problems, and bibliographic references. 

There are several other, more abstract, results pertaining to the existence of $\Wb$. The first is in \cite{DB}, which shows that there is an $\Wb$ attaining the infimum
\[
\l_k(\Wb) = m_k:= \inf \{\l_k(\W) : |\W|=1, \W\ss B_R \}.
\]
for every $R$ large enough (so that $|B_R|>1$); the resulting $\Wb$ is \emph{quasiopen}, which is a suitable relaxation of open sets for which $\l_k(\Wb)$ may still be defined, see Section \ref{sec:not}. This extra boundedness constraint was removed only recently, simultaneously in \cite{Bucur} and \cite{MP}. These works show that there is a bounded quasiopen set $\Wb$ satisfying \eqref{eq:mainprob}, and the first also shows that $\Wb$ has finite perimeter. It is unknown whether $\Wb$ admits an equivalent open representative, and indeed the only regularity result pertaining to this problem is \cite{BMPV}, which shows that there is an eigenfunction $u_k$ for $\Wb$ which is Lipschitz continuous.

Our goal in this paper is to provide a more detailed characterization of the boundary of $\Wb$. Here is our main theorem, written specifically for $\Wb$ satisfying \eqref{eq:mainprob}:

\begin{theorem} Let $\Wb$ be a quasiopen set satisfying \eqref{eq:mainprob}. Then there is a representative for $\Wb$ whose boundary admits the decomposition
	\[
	\p \Wb = \p^* \Wb \cup Z_{AC} \cup Z_{C},
	\]
	and has the following properties:
	\begin{enumerate}
		\item The set $\p^* \Wb$ is the reduced boundary; it is relatively open in $\p \Wb$, and it locally coincides with the graph of an analytic function.
		\item The set $Z_{AC}$ is a relatively open subset $\p \Wb \sm \p^* \Wb$ containing all of the points $x$ where
		\[
		\liminf_{r\searrow 0} \frac{|B_r(x) \cap \Wb| }{|B_r|} < 1.
		\]
		This set has a Hausdorff dimension of at most $n-3$.
		\item The set $Z_C$ consists of points $x$ in $\p \Wb$ with
		\[
		\lim_{r\searrow 0} \frac{|B_r(x) \cap \Wb| }{|B_r|} = 1.
		\]
		\item There is an orthonormal set of eigenfunctions $\{u_j\}$ for an open subset of $\Wb$, each with eigenvalue $\l_k(\Wb)$, which are all Lipschitz continuous. In addition, there are numbers $\xi_j$, with
		\[
		\sum_j \xi_j = 1,
		\]
		so that at any point $x\in \p^* \Wb$ we have
		\begin{equation}\label{eq:el}
		\sum_j \xi_j (u_j)_\nu^2 (x) = 1.
		\end{equation}
		Here $(u_j)_\nu$ represents the partial derivative of $u_j$ in the direction normal to $\Wb$ at $x$.
	\end{enumerate}
\end{theorem}

In other words, we show that $\p \Wb$ may be decomposed into three regions: the regular part of the boundary $\p^* \Wb$, which we show is smooth, a tiny set of singular points $Z_{AC}$, and some ``cusp'' points $Z_C$ at which the Lebesgue density of $\Wb$ is $1$. In addition, we obtain a kind of first variation condition on $\p^* \Wb$, giving a relationship between the normal derivatives of (possibly) several eigenfunctions of $\Wb$. While this may be interpreted as a first variation formula, it does not originate from performing shape variations directly on $\Wb$, and is not found in the literature. See \cite{H} for a discussion of why performing domain variations leads to underwhelming results for this particular problem. In fact, we would like to suggest that this formula \eqref{eq:el} is of at least equal interest to the accompanying regularity result, and feel that further understanding the constants $\xi_k$ would be of great interest.

The main weakness of our result, as far as the regularity aspect, is the lack of information about $Z_C$. We suspect that this set need not generally be empty, and discuss our expectations for the local behavior near such points in Section \ref{sec:reg}. This also seems consistent with the numerical studies; some relevant discussion can be found in \cite{AP}, for instance. However, our current methods are inadequate to provide further information here, and we leave this as another interesting open problem.

Before commenting on the proof, we note that we obtain a similar theorem for a much more general class of minimization problems. The precise assumptions are given in Section \ref{sec:not}, but they include minimizers, over all quasiopen sets $\W$, of
\[
F(\l_1(\W),\ldots,\l_N(\W)) + |\W|.
\]
Here $F$ should be $C^1$and nondecreasing in its parameters. We can also treat minimizers of
\[
F(\l_{k_1}(\W),\ldots,\l_{k_N}(\W)) + |\W|,
\]
where $F$ is Lipschitz continuous and strictly increasing in every parameter. This accounts for many natural examples. In particular, convex combinations of several sequential eigenvalues, like
\[
F = \a \l_k(\Wb) + \b \l_{k+1}(\Wb),
\]
have appeared several times in the numerical literature \cite{OK1,OK2}.

\begin{theorem}\label{thm:main} Let $F,\Wb$ satisfy (A1-4) and (B1-5) (see Section \ref{sec:not}). Then there is a representative for $\Wb$ whose boundary admits the decomposition
	\[
	\p \Wb = \p^* \Wb \cup Z_{AC} \cup Z_{C},
	\]
	and has the following properties:
	\begin{enumerate}
		\item The set $\p^* \Wb$ is the reduced boundary; it is relatively open in $\p \Wb$, and it locally coincides with the graph of an analytic function.
		\item The set $Z_{AC}$ is a relatively open subset $\p \Wb \sm \p^* \Wb$ containing all of the points $x$ where
		\[
		\liminf_{r\searrow 0} \frac{|B_r(x) \cap \Wb| }{|B_r|} < 1.
		\]
		This set has a Hausdorff dimension of at most $n-3$.
		\item The set $Z_C$ consists of points $x$ in $\p \Wb$ with
		\[
		\lim_{r\searrow 0} \frac{|B_r(x) \cap \Wb| }{|B_r|} = 1.
		\]
		\item There is an orthonormal set of eigenfunctions $\{u_j\}$ of a subset $\Wb$, with eigenvalues at most $\l_N(\Wb)$, which are all Lipschitz continuous. In addition, there are nonzero numbers $\xi_j$
		so that at any point $x\in \p^* \Wb$ we have
		\begin{equation}\label{eq:el2}
		\sum_j \xi_j (u_j)_\nu^2 (x) = 1.
		\end{equation}
		Here $(u_j)_\nu$ represents the partial derivative of $u_j$ in the direction normal to $\Wb$ at $x$.
	\end{enumerate}
\end{theorem}

Recently, we published a companion paper \cite{KL}, which addresses a similar problem but makes the extra assumption that $F$ depends on all of the eigenvalues $\l_1(\Wb),\ldots, \l_N(\Wb)$ for some $N$, and is uniformly increasing in every parameter. There we termed this the ``nondegenerate'' case, and developed an approximation argument and a free boundary method to prove a result similar to \ref{thm:main}, but with $Z_C$ empty. See also the paper \cite{MTV}, which proves a similar result with different techniques. This paper builds on the method in \cite{KL}, but removes the nondegeneracy assumption used there.

To explain briefly why the ``degenerate'' case treated here is (dramatically) harder than the nondegenerate one, there are two free boundary estimates which are typically available in this kind of problem, and which date back to the foundational work \cite{AC}. One is an ``upper bound'' on the eigenfunctions, here meaning that
\[
\max_{j, \R^n} |\n u_j|\leq C
\]
where the $u_j$ are as in Theorem \ref{thm:main}(4). We prove that here. The second, however, is a ``lower bound,'' which would say that for all $x\in \p \Wb$ and $r$ small,
\[
\max_{j,B_r(x)} | u_j| \geq c r.
\]
We prove an estimate like this in the nondegenerate case in \cite{KL}. Together, these estimates ensure that $\p \Wb$ is tame from a geometric measure theory perspective, and aid various blow-up and compactness arguments. They work to  ensure that information may be passed freely between the set $\p \Wb$ and the functions $u_j$. In the degenerate case, this lower bound is unavailable. Therefore, we have to use various unusual methods (often involving the Weiss formula) in place of standard free boundary techniques.

The general method of proof is based on solving some auxiliary minimization problems
\begin{equation}\label{eq:main}
\min \{ F_p(\l_1(\W),\ldots,\l_N(\W)) + |\W|  \},
\end{equation}
where $F_p$ tends to $F$ as $p\rightarrow \8$. $F_p$ is chosen so that it has an additional structural property: the minimizers $\W_p$ of \eqref{eq:main} all have $\l_1(\W_p)<\ldots<\l_{N+1}(\W_p)$; their eigenvalues are simple. This structural property means that we can successfully perform domain variation arguments to obtain a first variation formula for $\W_p$, analogous to \eqref{eq:el2}. We construct and study the sets $\W_p$ in Section \ref{sec:wp}.

The key step, then, is to pass to the limit. This requires obtaining uniform estimates on the $\W_p$, the most important of which (the uniform Lipschitz bound on the eigenfunctions) is shown in \ref{sec:lip}. There is a secondary estimate, due to D. Bucur \cite{Bucur}, which we discuss in Section \ref{sec:db}, which is needed for technical reasons. These are put to use in Section \ref{sec:lim}, where we pass to the limit and explain in what manner the $\W_p$ converge to the original minimizer $\Wb$. Further arguments, related to pushing the formula \eqref{eq:el2} to the limiting set, are presented in Section \ref{sec:info}. Finally the regularity aspects of Theorem \ref{thm:main} are proved in Section \ref{sec:reg}.

\section{Notation, Vocabulary, and Background}\label{sec:not}

The lowercase $n$ will refer to the dimension of the ambient Euclidean space $\R^n$ in which our problem is set, while the uppercase $N$ will denote the largest eigenvalue appearing in our minimization problems. For a measurable set $E$, $|E|$ denotes the Lebesgue measure. The Euclidean balls $B_r$ are centered around $0$, while $B_r(x)$ are centered around $x$. Constants will be denoted by $C$, and may vary from line to line; statements of theorems and lemmas will explain what the constants depend on, and this is only made explicit in the proof when there may be ambiguity or special emphasis is merited.

Any Sobolev function is identified with its quasicontinuous representative. Any quasiopen set is identified with a finely open representative $E$ with the property that if $\text{cap}(B_r(x)\sm E)=0$, then $B_r(x)\ss E$. The term quasi-everywhere, abbreviated q.e., means up to a set of capacity $0$. A quasiopen set will always assumed to be bounded, but we will not specify diameter explicitly unless relevant to the specific argument. For a quasiopen set $\W$,  the space $H^1_0(\W)$ is the space of functions $u\in H^1(\R^N)$ which are $0$ outside of $\W$ quasi-everywhere
% while the larger space $\tilde{H}^1_0$ consists of functions in $H^1(\R^N)$ which vanish outside of $\W$ almost everywhere.

The subscript $k$ will typically be reserved for enumerating the eigenvalues, and $\l_k(\W)$ denotes the $k$-th Dirichlet eigenvalue of the quasiopen set $\W$: in other words,
\[
\l_k(\W) = \inf_{E\ss H^1_0(\W), \dim E = k} \max_{v\in E\sm \{0\}} R[v],
\]
where $R$ is the Rayleigh quotient
\[
R[v] = \frac{\int_\W |\n v|^2 }{\int_\W v^2}.
\]
Given a (bounded) quasiopen set $\W$, one may always produce a collection of eigenfunctions $v_k$, each associated with the eigenvalue $\l_k(\W)$. Unless stated otherwise, any collection of eigenfunctions will be assumed to be orthonormal. Note that such a collection is not unique, and at times it will be important to select particular families of eigenfunctions; this will be detailed explicitly.

For the duration of the paper, we will fix a function $F : (0,\8)^N \rightarrow [0,\8)$. Here are the standing assumptions placed on $F$:
\begin{enumerate}[({A}1)]
	\item $F$ is (locally) continuous.
	\item $F$ is nondecreasing in each parameter:
	\[
	F(\k_1+p_1,\ldots,\k_N + p_N) \geq F(\k_1,\ldots,\k_N) \qquad \forall \{p_k\}, p_k\geq 0.
	\]
	\item $F$ is (locally) Lipschitz continuous.
	\item Say that for some $\k\in (0,\8)^N$ and $t>0$, $F$ has $F(\k) = F(\k_1,\ldots, \k_k - t,\ldots,\k_N)$. Then
	\[
	\lim_{h\searrow 0} \frac{1}{h^{N-1}}\int_{[0,h]^{N-1}} |F(\k_1 +\z_1,\ldots,\k_k + h, \ldots , \k_N +\z_N) - F(\k_1 +\z_1,\ldots,\k_k ,\ldots, \k_N +\z_N)| d\z_1\ldots d\z_N = 0.
	\] 
\end{enumerate}
We will also fix a quasiopen set $\Wb$, which has the following properties:
\begin{enumerate}[({B}1)]
	\item $\Wb$ minimizes
	\[
	F(\l_1(\W),\ldots,\l_N(\W)) + |\W|
	\]
	over all quasiopen sets $\W$.
	\item $\l_N(\W)<\l_{N+1}(\W)$, and either $F$ does not depend on $\l_k$ or $\l_k<\l_{k+1}$ for every $k < N$.
	\item For any quasiopen $\W$ satisfying (B1) with $|\W \sm \Wb|=0$, we have $\text{cap}(\W\sm \Wb)=0$.
	\item $\Wb$ is bounded.
	\item $|\Wb|>0$.
%	\item $|\p \Wb|=0$.
\end{enumerate}

Some comments about these assumptions are in order:
\begin{enumerate}
	\item We will write $F(\W)$ for $F(\l_1(\W),\ldots,\l_N(\W))$. The terms \emph{minimizer of} $F$ and $F$ \emph{minimizer} will be used to refer to sets minimizing $F(\W)+|\W|$ over quasiopen sets.
	\item Given a function $F$ and a set $\Wb$ satisfying (B1-5), it is enough that in place of (A1) $F$ is just continuous at the point $(\l_1(\Wb),\ldots,\l_N(\Wb))$. This is because if $\Wb$ is an $F$ minimizer, it is also an $F'$ minimizer for any $F'$ which has $F'(\Wb)=F(\Wb)$ and $F'\geq F$ pointwise. It is then possible to replace $F$ by a larger $F'$ which agrees with $F$ on the spectrum of $\Wb$ but satisfies (A1) globally. A typical example of such an $F'$ is the following sup-convolution, where $\w$ is a modulus of continuity such that $|F(\k) - F(\Wb)|\leq \w(|\k - (\l_1(\Wb),\ldots,\l_N(\Wb))|)$:
	\[
	F'(\k) = \sup_{p_k\geq 0} F(\k + p) - \w(|p|).
	\]
	A similar statement is true for (A3).
	\item The assumptions (A1-4) are rather general, and include many examples of special interest. The book \cite{H} discusses some of these examples, including $F$ which are linear, or are linear in the powers (or reciprocals), of the eigenvalues. The key point we wish to emphasize here is that unlike in our earlier work \cite{KL}, here $F$ may depend on just some subset of the eigenvalues, and not depend on some of them at all. The special case $F(\W)=\l_N(\W)$ is the canonical example of this, and the reader might find it useful to focus on this example for concreteness. This case already captures most of the challenges of the general problem, and our results are as new for that particular $F$ as for any other.
	\item When showing existence of minimizers, usually an assumption like (A2), or sometimes even a stronger version of it, is required (although see \cite{BBF}). Omitting it dramatically changes the character of the problem at a global level. Assumption (A1) is not necessary for showing existence of minimizers; usually lower semicontinuity is all that is required. However, (A1) appears necessary for regularity, at least when using only local arguments. One way to see this is to consider any set $\Wb$ which minimizes $F(\W)+|\W|$ over all competitors $\W$ which have $\Wb\ss \W$, for some smooth $F$: then setting
	\[
	\begin{cases}
	 F'(\k) = F(\k) + C & \k_k > \l_k(\Wb)\\
	 F'(\k) = F(\k) & \text{otherwise,}
	\end{cases}
	\] 
	for a sufficiently large $C$ (say larger than $|\Wb|$) we have that $\Wb$ minimizes $F'$ over all $\W$ with $\W \ss \Wb$ as well. However, such one-sided minimizers generally need not have smooth boundary.
	\item Clearly  (A3) implies (A1). We separate these assumptions because some of our arguments only require (A1-2), not (A3-4), and the reason why we need (A3) at all is somewhat technical. It would be reasonable to suppose that (A3) is not actually required for regularity of $\p \Wb$.
	\item The assumption (A4) says essentially that at points where $F$ is constant in the ``backward'' direction in one variable, the partial derivative of $F$ in that variable is approximately continuous and has value $0$. There are two useful situations in which a function would automatically satisfy (A4): any locally $C^1$ function $F$ will satisfy it, but so will any $F$ which, for each $k$, is either strictly increasing in $\k_k$ or does not depend on $\k_k$ at all. This second situation seems to cover most of the common examples.
	\item In the shape optimization literature, it is more common to consider the volume-constrained problem
	\[
	\min_{|\W|=1} F(\W).
	\]
	Provided $F$ is a homogeneous function, the two problems are equivalent up to rescaling.
	\item The existence of an $\Wb$ satisfying (B1) is very much nontrivial, and we do not investigate it here. The works \cite{Bucur,MP} both prove that minimizers (over the class of arbitrary quasiopen sets) exist and are bounded, giving (B4) as well. 
	\item The assumption (B2) may be made without any loss of generality, and is just a notational convenience. Indeed, if $\l_N(\Wb) = \l_{N+1}(\Wb) = \ldots = \l_{N'}(\Wb)<\l_{N'+1}(\Wb)$, we simply replace $N$ by $N'$, and extend $F$ to not depend on the newly added eigenvalues. Likewise, if $\l_k = \l_{k+1}$ for some $k<N$, we replace $F$ by $F'(\l_1,\ldots,\l_{k},\l_{k+1},\ldots,\l_N)=F(\l_1,\ldots,\l_{k+1},\l_{k+1},\ldots,\l_N)$. Notice this has nothing in common with assuming that $\l_k<\l_{k+1}$ for all (or some) $k<N$, which is a very strong assumption conjectured to be never valid at least in some cases (see \cite{H}). 
	\item Whether the assumption (B3) is true in general is unclear. However, if we set $\mathcal{U}$ to be the collection of all quasiopen $U$ (it is important to take finely open representatives here) which have $|U\sm \Wb|=0$, then there exists a countable sub-collection $\mathcal{U}_0$ such that
	\[
	\text{cap}(\cup_{U\in \cU} U \sm \cup_{U\in \cU_0}U) = 0;
	\]  
	see \cite{MZ}[Theorem 2.146]. Setting $\W = \cup_{U\in \cU_0}U$, we have that $\text{cap}(\Wb\sm \W)=0$ and $|\W \sm \Wb|=0$. These properties imply that $\W$ is also an $F$ minimizer, and that it satisfies (B3).	Thus a reasonable approach to studying an arbitrary minimizer $\Wb$ is to first replace it by $\W$ and study that, and then return to $\Wb$. This replacement is analogous to replacing a nice open set by the interior of its closure. See Remark \ref{rem:smallermin} for additional discussion.
	\item Finally, while the assumption (B5) appears self-explanatory, we note that this is the only non-degeneracy assumption that we make. It would perhaps be reasonable to expect that some strict, uniform, version of (A2) would be needed for regularity purposes. It turns out, however, that any $F$ will only admit nontrivial minimizers with spectrum located at points where $F$ is strictly increasing. That is demonstrated in the following lemma.
\end{enumerate}

\begin{lemma}\label{lem:strictinc} Let $F$ satisfy (A1-2), and $\Wb$ satisfy (B4-5).
	\begin{enumerate}
		\item If $\Wb$ also satisfies (B1), then
			\begin{equation}
			K^{-1}\leq \l_1(\Wb)\leq \cdots \leq \l_N(\Wb)\leq K.
			\end{equation}
			Here $K$ depends on $n,N,F,|\Wb|$, and $\diam \Wb$.
		\item If $\Wb$ has the property that for any $\W$ with $\l_k(\W)\geq \l_k(\Wb)$, $k\in [1,N]$,
		\begin{equation}\label{eq:strictinch1}
		F(\Wb) \leq F(\W) - \frac{1}{4} (|\Wb| - |\W|),
		\end{equation}
		then there are numbers $c,s_0$ such that for any $s\in (0,s_0)$, we have
		\begin{equation}\label{eq:strictinc}
		\frac{F(\l_1(\Wb)+s,\ldots, \l_N(\Wb)+s) - F(\Wb)}{s}\geq c.
		\end{equation}
		The values of $c$ and $s_0$ will depend on $n,N,|\Wb|$, and $\l_N(\Wb)$.
		\item If $\Wb$ has the property that for any $\W$ with $\l_k(\W)\leq \l_k(\Wb)$, $k\in [1,N]$,
		\begin{equation}\label{eq:strictinch2}
		F(\Wb) \leq F(\W) + 4 (|\W| - |\Wb|),
		\end{equation}
		then for any $s\in (0,s_0)$, we have
		\begin{equation}\label{eq:lip}
		\frac{F(\l_1(\Wb)-s,\ldots, \l_N(\Wb)-s) - F(\Wb)}{-s}\leq c.
		\end{equation}
		Here $c$ and $s_0$ are as in (2).
	\end{enumerate}
\end{lemma}

\begin{proof}
	First, for (1) note that $|\Wb|\leq C(\diam \Wb)^n \leq C$; this and the Faber-Krahn inequality imply that $\l_1(\W)\geq C$. To bound $\l_N(\Wb)$, we recall that for any quasiopen set $E$, the ratio $\l_N(E)/\l_1(E)\leq C=C(n,N)$ (see \cite{MP}). There are two cases now: if $F(t,\ldots,t)\rightarrow \8$ as $t\rightarrow \8$, it follows immediately that $F(\Wb) \leq F(B_1) +|B_1|$, which controls  $\l_N(\Wb)$. If instead $F(t,\ldots,t)\rightarrow F_\8<\8$ as $t\rightarrow \8$, we have that
	\[
	\inf_\W F(\W) + |\W|\leq F_\8
	\] 
	by using smaller and smaller balls as competitors. This gives $F(\Wb)\leq F_\8 - |\Wb|$, which controls $\l_N(\Wb)$.
	
	For (2), consider a dilation $t\Wb$ as a competitor for $\Wb$, with $t<1$. Then we have
	\[
	F(\Wb) \leq F(t\Wb) + \frac{1}{4}(t^n - 1) |\Wb|\leq F(\l_1(\Wb)+s,\ldots,\l_N(\Wb)+s) + \frac{1}{4}(t^n -1) |\Wb|
	\]
	for $s= (t^{-2}-1)\l_N(\Wb)$, using \eqref{eq:strictinch1}  and then (A2). For $t$ close enough to $1$ (equivalently, $s\in (0,s_0)$), we have $1-t^n\geq c(1-t)\geq c s$, which implies \eqref{eq:strictinc}.
	
	The proof of (3) is identical, just using $t>1$ instead.
\end{proof}

\section{The Regularized Problem}\label{sec:wp}

Our analysis of the regularity of $F$ minimizers is based on solving approximate optimization problems, which are described in this section. They can be interpreted as a kind of regularization of the original problem. However, this analogy should be used cautiously: while these approximate problems do fall into the "nondegenerate" class considered in \cite{KL}, the key feature they exhibit is the presence of an Euler-Lagrange equation, or first variation formula, which cannot easily be verified in the original problem by direct domain variations. This is not really a statement of regularity, but rather a structural aspect of the approximation.

Let $G_p$ be the following averaged function:
\begin{equation}\label{eq:gp}
G_p(\k) = p^N \int_{[0,\frac{1}{p}]^N} F(\k + \z)d\z.
\end{equation}
This function has $G_p \geq F_p$, satisfies (A1-2), and has the additional property that it is continuously differentiable (but with bounds which do not remain uniform as $p\rightarrow \8$).

Set 
\begin{equation}\label{eq:tau}
\t_{k,p}=\t_{k,p}(\k_1,\ldots,\k_k) = \1\sum_{l=1}^k \k_{l}^p\2^{\frac{1}{p}},
\end{equation}
and then
\begin{equation}\label{eq:fp}
F_p(\k_1,\ldots,\k_N) = G_p (\t_{1,p},\t_{2,p},\ldots, \t_{N,p}) + \frac{1}{p} \sum_{k=1}^N (N+1-k)\k_k.
\end{equation}
The second term is a small perturbation, and helps avoid needless technical difficulties later. Fix $\chi$ to be a fixed smooth function which has $\chi(0)=0$, $\chi(t)> 0$ for $t\neq 0$, $-\frac{1}{2}\leq \chi' \leq \frac{1}{2}$, and $\chi'(0)=0$. We will be interested in sets $\W_p$ which minimize, for a sufficiently small but fixed $s$,
\[
F_p(\W) + |\W|+E(\W),
\]
where
\[
E(\W) =  \int_\W s \max \{d(x,\Wb),1\} + \int_{\W^c} s \max \{d(x,\Wb^c),1\} + \chi(|\Wb|-|\W|).
\]
 Such sets will be referred to as \emph{$p$-minimizers} below.
 
Let us record here the derivatives of $F_p$ with respect to each parameter:
\begin{equation}\label{eq:fpdir}
\p_{\k_k} F_p(\k) =  \frac{N+1-k}{p} +   \sum_{j=k}^N \1 \frac{\k_k}{\t_{k,p}} \2^{p-1} \cdot \p_{\k_j} G_p(\k).
\end{equation}

The most useful sense in which our various minimizers will satisfy Euler-Lagrange equations is in the viscosity sense, defined below.
\begin{definition}Given an open set $U$, a finite set $\{v_k\}_{k=1}^N$ of Lipschitz functions defined on $U$ and an open set $\W\ss U$ outside of which they vanish is said to be a \emph{viscosity solution with parameters} $\{\xi_k\}_{k=0}^N, C_0$ if it satisfies the conditions below. The $\xi_k$ should have $\xi_k \geq 0$ and constant for $k>0$, and $\xi_0$ a Lipschitz function taking values in $(\frac{1}{2},\frac{3}{2})$. 
	\begin{enumerate}[({V}1)]
		\item $|\triangle v_k|\leq C_0$ on $\W\cap U$ for each $k$.
		\item If there is a ball $B_r(y)\ss \W$ with $B_{2r}(y)\ss U$ such that $\{x\} =  \p B_r(y)\cap \p \W$, and a sequence of numbers $\b_k\geq 0$ such that
		\[
		|u_k(z)| \geq \b_k (\frac{y-x}{|y-x|}\cdot (z-x))^+ + o(|z-x|),
		\]
		then
		\[
		\sum_{k=1}^N \xi_k \b_k^2 \leq \xi_0(x).
		\]
		\item If there is a ball $B_r(y)\ss U\sm \W$  with $B_{2r}(y)\ss U$ such that $\{x\} =  \p B_r(y)\cap \p \W$, and a sequence of numbers $\b_k\geq 0$ such that
		\[
		|u_k(z)| \leq \b_k (\frac{y-x}{|y-x|}\cdot (z-x))^- + o(|z-x|),
		\]
		then
		\[
		\sum_{k=1}^N \xi_k \b_k^2 \geq \xi_0(x).
		\]
		\end{enumerate}
\end{definition}

Many properties of $p$-minimizers follow directly from our earlier work \cite{KL} or slight modifications of arguments there, and we begin exploring them in the propositions below. Most of these properties will not be uniform in $p$.

\begin{proposition}\label{prop:fp1} Let $F_p$ be as in \eqref{eq:fp}, and assume that $F$ satisfies (A1-2) while $\Wb$ satisfies (B1-5).
\begin{enumerate}
	\item $F_p$ is continuously differentiable in its parameters, and the partial derivatives $\p_{\k_k}F$ are all strictly positive. Moreover, the $F_p$ satisfy (A1-2). If $F$ satisfies (A3), then so does $F_p$, with Lipschitz constants on compact subsets of $(0,\8)^N$ bounded uniformly in $p$. 
	\item For any quasiopen set $\W$, $F_p(\W)\geq F(\W)$, and $\lim_{p\rightarrow \8} F_p(\W) = F(\W)$. 
	\end{enumerate}	
\end{proposition}

\begin{proof}
	For (1), this follows from \eqref{eq:fpdir}, noting that the first term is strictly positive, while the ones in the sum are nonnegative. That $F_p$ satisfies (A1-2) (and (A3) when $F$ does) comes from the definition \eqref{eq:fp}. In (2), the first is clear from the definitions, while the second follows as $\t_{k,p}\rightarrow \k_k$ as $p\rightarrow \8$.	
\end{proof}

\begin{lemma}\label{lem:ex} Let $F_p$ be as in \eqref{eq:fp}, $p\geq p_0$, and assume that $F$ satisfies (A1-2) while $\Wb$ satisfies (B1-5). Then:
	\begin{enumerate}
		\item There exists a quasiopen $p$-minimizer $\W_p$ for each $p$. 
		\item For this $\W_p$, $\l_N(\W_p)\leq C$, with the constant depending on $F$ and $\Wb$ but not $p$.
		\item This $\W_p$ also has $c \leq |\W_p| \leq C$, with the constants depending on $F$ and $\Wb$ but not $p$.
		\item We have
		\begin{equation}\label{eq:exc1}
		E(\W_p) + |F_p(\W_p) - F_p(\Wb)| + |F(\W_p) - F(\Wb)| \rightarrow 0
		\end{equation}
		as $p\rightarrow 0$.
	\end{enumerate} 
\end{lemma}

\begin{proof}
	First, we will prove a stronger version of (2) and (3). Let us set
	\[
	m(p) = \inf \{F_p(\W)+|\W|+ E(\W): \W \text{ quasiopen}\}. 
	\] 
	Then we claim, for each $\e>0$, that if $p\geq p_0$ and for some quasiopen set $\W$ we have $F_p(\W)+|\W|+ E(\W) - m(p) \leq \d$ for a sufficiently small $\d>0$, then $\l_N(\W) \leq C_1$ and $||\W|-|\Wb||\leq \e$ (here $C_1$ and $\d$ depend only on $F,\e$ and $\Wb$, but not $p$).
	
	Indeed, we know that
	\begin{align}
	F(\Wb) + |\Wb| + E(\W) & \leq F(\W) + |\W| + E(\W)\nonumber\\
	& \leq F_p(\W) + |\W|+E(\W)\nonumber\\
	& \leq F_p(\Wb) + |\Wb| + E(\Wb) + \d. \label{eq:exi5} 
	\end{align}
	Here the first line used the minimality of $\Wb$, the second line that $F_p \geq F$ (from Proposition \ref{prop:fp1}), and the third the definition of $\d$. Now if we use Proposition \ref{prop:fp1} as well as that $E(\Wb)=0$ we obtain that
	\begin{equation} \label{eq:exi3}
	F(\Wb) + |\Wb| + E(\W) \leq F(\W) + |\W| + E(\W) \leq F(\W) + |\Wb| + \d + o_p(1), 
	\end{equation}
	so 
	\begin{equation}\label{eq:exi1}
	E(\W) + |F(\W) - F(\Wb) + |\W| - |\Wb| | \leq \d + o_p(1).
	\end{equation}
	From the first term, this gives
	\begin{equation}\label{eq:exi4}
	\chi(|\Wb|-|\W|) \leq \d + o_p(1),
	\end{equation}
	which implies $||\Wb|-|\W|| \leq \e$ if $\d$ is sufficiently small. Then using \eqref{eq:exi1} again, we see that for $p$ large enough,
	\begin{equation}\label{eq:exi2}
	|F(\Wb) - F(\W)|\leq 2\d + \e.
	\end{equation}
	
	Assume for contradiction that $\l_N(\W) > C_1$ for some $C_1$ large. Then from the Ashbaugh-Benguria inequality (see \cite{MP}[Theorem B] for a sufficiently general version), we have $\l_k(\W)\geq c C_1$ for all $k\in [1,N]$ and a $c$ depending only on $N$ and $n$. Let us select $C_1$ so that $c C_1 \geq \l_k(\Wb) + 1$ for all $k\in [1,N]$; then applying Lemma \ref{lem:strictinc} we see that
	\[
	F(\W) \geq F(\l_1(\Wb)+1,\ldots,\l_N(\Wb)+1) \geq F(\Wb) + c_1
	\]
	for some $c>0$. Choosing $\d$ so that $2\d + \e < \frac{c_1}{2}$ gives a contradiction to \eqref{eq:exi2}, proving the claim.
	
	We will now sketch the proof of (1). Fixing $p>p_0$, our claim implies that it suffices to minimize $F_p$ over the class of sets $\W$ which have $\l_N(\W)\leq C_1$. Moreover, the volume bound in the claim and the Faber-Krahn inequality also imply that we may as well restrict to sets with $\l_1(\W) \geq c$, for some $c$ depending only on $|\Wb|$; set
	\[
	\cQ = \{\W \text{ quasiopen}: c\leq \l_1(\W), \l_N(\W)<C_1 \}.
	\]
	Using Proposition \ref{prop:fp1}, part (1), we see that $0<c(p)\leq \p_{\l_k}F(\W)\leq C(p)<\8$ for all $\W\in \cQ$.
	
	Let us find (for every $R$ large) $\W^R\ss B_R$ for which
	\[
	F_p(\W^R)+|\W^R|+ E(\W^R) = \inf \{F_p(\W)+|\W|+ E(\W) :\W \in \cQ , \W \ss B_R \}.
	\] 
	That such sets exist follows from \cite{DB}. We may apply \cite{KL}[Lemma 2.1] to these $\W^R$, and then argue as in \cite{KL}[Lemma 4.2] to show that up to selecting different $\W^R$, we may take $\W^R$ to be bounded uniformly in $R$ (by a constant $M$ that may depend on $p$). Using the compactness and lower semicontinuity results of \cite{DB} again, we  then take a limit, in the weak $\g$ sense, of $\W^R\rightarrow \W_p\ss B_M$.  This gives that
	\[
	F_p(\W_p)+|\W_p|+ E(\W_p) = \inf \{F_p(\W)+|\W|+ E(\W) :\W \in \cQ \} = m(p),
	\] 
	giving (1) (and also (2,3)).
	
	To see (4), we insert $\W = \W_p$ into \eqref{eq:exi1} and \eqref{eq:exi4}, with $\d=0$. This shows that
	\[
	\lim_{p\rightarrow \8} E(\W_p) + |F(\Wb)- F(\W_p)| + ||\Wb| - |\W_p|| = 0.
	\]
	Combining with \eqref{eq:exi5} gives \eqref{eq:exc1}.
\end{proof}

We continue with some further properties of the $\W_p$ constructed in Lemma \ref{lem:ex}, most of which are not uniform in $p$ (except as noted):

\begin{proposition}\label{prop:fp2} Let $F_p$ be as in \eqref{eq:fp}, $p>p_0$, and $F$, $\Wb$, $\W_p$ as in Lemma \ref{lem:ex}.
	\begin{enumerate}
		\item $\W_p$ is an open set.
		\item The first $N$ eigenfunctions $u_{k,p}$ of $\W_p$ are Lipschitz continuous, with constant $C=C(p)$. They also satisfy the lower bound: for any $x\in \W_p$,
		\[
		\sup_{k, y\in B_\r(x)}|u_{k,p}(y)| \geq c(p) \r
		\] 
		for every $\r<\r_0(p)$. Neither is claimed to be uniform in $p$.
		\item The partial derivatives $\p_{\k_k}F_p$ have the property that $\p_{\k_k}F_p(\k)>\p_{\k_{k+1}}F_p(\k)$ if $\k_k=\k_{k+1}$. It follows that $\l_1(\W_p)<\ldots<\l_{N+1}(\W_p)$ for any $p$-minimizer.
		\item Let $\xi_{k,p} = \p_{\k_k}F_p(\W_p)$ and $\xi_{0,p}(x) = 1 + s \max \{d(x,\Wb),1\}) - s \max \{d(x,\Wb^c),1\} + s \chi'(|\Wb|-|\W_p|)$. Then the set $\W_p$ and functions $\{u_{k,p}\}_{k}$ constitute a viscosity solution (with $U=\R^n$) with parameters $\{\xi_{k,p}\}, \max_{k,\W_p}\l_k(\W_p)|u_{k,p}|$.
		\item The Weiss energies
		\[
		W_p(x,r) = \frac{1}{r^n} \int_{B_r(x)\cap \W_p} (\sum_{k=1}^N \xi_{k,p}|\n u_{k,p}|^2 + \xi_{0,p}) - \frac{1}{r^{n+1}}\int_{\p B_r(x)} \sum_{k=1}^N \xi_{k,p}u_{k,p}^2 d\cH^{n-1}
		\]
		satisfy, for $s<r$ and $x\in \p \W_p$,
		\[
		W_p(x,r) - W_p(x,s) \geq \int_s^r \frac{2}{t^{n+2}}\int_{\p B_t(x)} \sum_{k=1}^N \xi_{k,p}(u_{k,p} - t (u_{k,p})_r)^2 - C(r-s),
		\]
		with $C$ independent of $p$.
		\item The reduced boundary $\p^* \W$ is composed of a union of $C^{1,\a}$ graphs, while the complement $\p \W\sm \p^* \W$ has Hausdorff dimension at most $n-5$. 
		\item  We have
		\[
		\sum_{k=1}^N \xi_{k,p} (u_{k,p})_\nu^2 (x) = \xi_{0,p}(x)  
		\]
		for every $x\in \p^*\W$. Here $\nu$ is the unit normal vector at $x$ to $\W_p$.
	\end{enumerate}	
\end{proposition}

\begin{proof}[Sketch of proof.]
		The openness of $p$-minimizers can be deduced from \cite{KL}[Lemma 2.5], which applies to $p$-minimizers. Property (2) comes from the main result in \cite{BMPV}, for the continuity, and \cite{KL}[Lemma 2.1], for the lower bound. At this point, it helps to observe that while Lemmas 2.6-2.7 of \cite{KL} are not valid for $p$-minimizers, the argument in Lemma 2.6 still goes through to show that $\W_p$ cannot have two different connected components with a common boundary point. Using this, all of the results of Section 3 apply equally well to $p$-minimizers.
		
		Property (3) comes from inspecting the formula \eqref{eq:fpdir}, while the consequence is from \cite{KL}[Remark 3.5]. We may now obtain (4) from \cite{KL}[Section 3], using that $\W_p$ has Property S. The monotonicity of the Weiss energy (5) may be verified as in the proof of \cite{KL}[Proposition 8.1]; the worse estimate on the error comes from the fact that $\xi_0$ is Lipschitz rather than constant, but the $C$ is uniform in $p$.
		
		Then (6) is the main result of \cite{KL}; the regularity of $\p^* \W_p$ follows from (4), (2), and \cite{KL}[Theorem 7.2], while the estimate on the size of $\p \W_p \sm \p^* \W_p$ uses \cite{KL}[Section 8] and (5). Finally, (7) is immediate from (6) and (4).
\end{proof}

We now turn our attention to estimates on $\W_p$ which are uniform in $p$. Let us start with a straightforward observation:

\begin{lemma}\label{lem:moreest} Let $F_p$ be as in \eqref{eq:fp}, $p>p_0$, and $F$, $\Wb$, $\W_p$ as in Lemma \ref{lem:ex}. The sets $\W_p$ enjoy the following properties, all with constants uniform in $p$:
	\begin{enumerate}
		\item \eqref{eq:lip} applies to $F_p$ and $\W_p$, with uniform constants. As a consequence,
		\[
		\sum_{k=1}^N \xi_{k,p} \leq c.
		\]
		\item \eqref{eq:strictinc} applies to $F_p$ and $\W_p$, with uniform constants. As a consequence,
		\[
		\sum_{k=1}^N \xi_{k,p} \geq c,
		\]
		so at least one of the $\xi_k$ is nonzero.
		\item $|u_{k,p}|\leq C$ for $1\leq k\leq N$
	\end{enumerate}
	Here $\xi_{k,p} = \p_{\k_k}F_p(\W_p)$.
\end{lemma}

\begin{proof}
	For properties (1) and (2), it is clear from the definition of $p$-minimizer that $\W_p$ and $F_p$ satisfy \eqref{eq:strictinch1} and \eqref{eq:strictinch2}, and so Lemma \ref{lem:strictinc}, (2-3), applies to $\W_p$ and $F_p$. The constants there depend only on quantities which are independent of $p$, using Lemma \ref{lem:ex}(2-3). The consequences follow by sending $s$ to $0$.
	
	On the other hand, (3) is a universal property of the eigenfunctions of a domain, with $C$ depending only on $\l_k(\W_p)$ (see \cite{BMPV}[Remark 2.5], for example). Using Lemma \ref{lem:ex} part (2), this is uniform in $p$.
\end{proof}

Noting that the numbers $\xi_{k,p}$ are bounded uniformly in $p$, we pass to a subsequence of $p$ (without changing notation) such that
\[
\lim_{p\rightarrow \8}\xi_{k,p} = \xi_k
\] 
for some $\xi_k$, for every $k\in [1,N]$.

\section{The Estimate of Dorin Bucur}\label{sec:db}

The two uniform estimates on $\W_p$ which would be most helpful to us are the lower bound and Lipschitz estimate of \ref{prop:fp2} part (2), but uniform in $p$. While we will take up the Lipschitz estimate in the following section, we do not know how to obtain the uniform lower bound (this is one of the central difficulties of these kinds of spectral optimization problems). Instead, we will use the groundbreaking approach of Dorin Bucur from \cite{Bucur}, who observes that a lower bound is easier to obtain for the torsion function on a domain, and this can still be of some use. We briefly summarize the method, and explain how it applies to $p$-minimizers.

The \emph{torsion energy} of an (quasi)open set $A$ is given by
\[
T(A) = \inf_{v\in H^1_0(A)} \int_A \frac{1}{2}|\n v|^2 - v.
\]
The infimum is attained by a function $v_A$ which satisfies the torsion equation on $A$: $-\triangle v_A = 1$ on $A$; we may extend $v_A$ by $0$ to $\R^n$. Let
\[
d_\g (A,B) = \int |u_A - u_B|
\]
be the distance associated with $\g$-convergence.

A \emph{torsion shape subsolution} with constants $\L,\d$ is a quasiopen set $A$ which has the property that given any open $B\ss A$ with $d_\g(A,B)<\d$, we have
\[
T(A) \leq T(B) + \L (|B|-|A|).
\]
Bucur established the following Alt-Caffarelli-type lower bound for torsion shape subsolutions, which also implies a diameter bound:

\begin{proposition}\label{prop:Bucur}
	Let $A$ be a torsion shape subsolution with constants $\d,\L$. Then there are $r_0$, $C_0$ depending on $\L,\t$, and $n$, such that if for $r\leq r_0$
	\[
	\fint_{B_{r}(x)} v_A \leq C_0 r,
	\] 
	then $v_A$ is zero on $B_{r/4}(x)$. Moreover, we have that each connected component of $A$ has diameter bounded by $C_0$.
\end{proposition}

The $L^1$ version follows from the $L^\8$ version by noting that $v_A$ solves $-\triangle v_A \leq 1$ in the sense of distributions on $\R^n$, and using the mean value property.

This is complemented by the following observation:

\begin{lemma}\label{lem:torsion} Assume that $F$ and $\Wb$ satisfy (A1-3) and (B1-5). Let $F_p$ be as in \eqref{eq:fp}, $p>p_0$, $\W_p$ as in Lemma \ref{lem:ex}.  The sets $\W_p$ (as well as $\Wb$) are torsion shape subsolutions, with constants $\L,\d$ depending only on $F,n$, and $N$, but not $p$. 
\end{lemma}

\begin{proof}
	We follow the proof in Bucur's \cite{Bucur}[Theorem 3.1], noting only that the constants $c_k(A)$ there depend only on $N$ and $\l_N(A)$ (as can be readily checked in the proof in the appendix), and we have that $\l_N(\W_p),\l_N(\W)\leq C$ from Lemmas \ref{lem:ex}(2) and \ref{lem:strictinc}(1). This gives
	\[
	\max_{k\leq N} \l_k(A) - \l_k(\W_p) \leq C (T(A) -T(\W_p)) \leq C d_\g(A,\W_p).
	\]
	for any $A\ss \W_p$ with $d_\g(A,\W_p)\leq \d$ with $\d$ not depending on $p$. On the other hand, after possibly taking $\d$ smaller, we may use (A3) from Proposition \ref{prop:fp1}(1) to write
	\[
	F_p(A) - F_p(\W_p) \leq C \max_{k\leq N} \l_k(A) - \l_k(\W_p).
	\]
	Combining, this gives
	\[
	F_p(A) - F_p(\W_p)\leq C (T(A) -T(\W_p)),
	\]
	so that from the minimality of $\W_p$ we get
	\[
	\frac{1}{4} |\W_p\sm A| \leq F_p(A) - F_p(\W_p)\leq C (T(A) -T(\W_p)).
	\]
	This implies that $\W_p$ is a torsion shape subsolution with constants uniform in $p$. The same argument works for $\Wb$ in place of $\W_p$.
\end{proof}

One application of this is a uniform bound on the diameter of $\W_p$. This is not the only way of obtaining this kind of result, and a different argument not requiring (A3) will be given in the next section.

\begin{corollary}\label{cor:uniformbd} Let $F,F_p,\Wb,\W_p$ be as in Lemma \ref{lem:ex}, and also that $F$ satisfies (A3). Then $\W_p$ has diameter bounded uniformly in $p$.
\end{corollary}

\begin{proof}
	This follows immediately from Lemma \ref{lem:torsion} and that $E(\W_p)\rightarrow 0$ from Lemma \ref{lem:ex}(4). Take $B_R$ to be a ball compactly containing $\Wb$; then as $E(\W_p)\rightarrow$, we have that $|\W_p \sm B_R|\rightarrow 0$. Consider a connected component $U$ of $\W_p$, then. From the Faber-Kahn inequality, we have that $\l_1(U) \geq C |U|^{-2/n}$. On the other hand, $\l_1(U)\leq \l_N(\W_p)$, for otherwise removing $U$ from $\W_p$ will not change the first $N$ eigenvalues but decrease the volume, contradicting the fact that $\W_p$ is a $p$-minimizer. Using Lemma \ref{lem:ex}(2), this gives $|U|\geq c$ uniformly in $p$, and so for $p$ large $U$ must intersect $B_R$. Using Lemma \ref{lem:torsion}, we learn that $U\ss B_{R+C_0}$.
\end{proof}

The more important corollary for us is the following technical result. The topological boundary $\p \Wb$ is difficult to deal with from a geometric measure theory perspective, but the shape subsolution property does yield some information about it.

\begin{corollary}\label{cor:smallbdry} We have $|\p \W \sm \W|=0$ for any torsion shape subsolution $\W$.	
\end{corollary} 

\begin{proof}
	Take any $x\in \R^n$: then using the mean value property,
	\[
	v_\W (x)\leq \fint_{B_r(x)} v_\W \leq C \frac{|B_r(x)\cap \W|}{r^n} \max_{B_r(x)} v_\W.
	\]
	Set $\phi(x,r)=\max_{B_r(x)}v_\W$; we have the relation
	\[
	\phi(x,r) \leq C_1 \frac{|B_{2r}(x)\cap \W|}{r^n} \phi(x,2r).
	\]
	Note also that $\phi(x,r)\leq C$. Now, take any point $x$ which has Lebesgue density $0$ for $\W$:
	\[
	\lim_{r\searrow 0} \frac{|\W\cap B_r|}{|B_r|} = 0.
	\]
	For all $r<r_1=r_1(x,\e)$, we then have $C_1|\W\cap B_{2r}|\leq \e r^n$, so
	\[
	\phi(x,r)\leq \e \phi(x,2r).
	\]
	It follows that
	\[
	\phi(x,r)\leq C (\frac{r}{r_1})^{-C\log {\e}} \leq C r^{2}
	\]
	for sufficiently small $\e$. Applying Proposition \ref{prop:Bucur}, it follows that $v_\W$ is zero on a small ball around $x$; this implies that the set of such $x$ has zero capacity. To summarize, we have shown that $\p \W$ consists of points of positive Lebesgue density for $\W$, up to a polar set.
	
	If $|\p \W \sm \W|>0$, then almost every point in this set has Lebesgue density $1$ for $\p \W \sm \W$, and so must have Lebesgue density $0$ for $\W$. We have just shown that the set of such points is polar, so it must be that $|\p \W \sm \W|=0$. 
\end{proof}

\begin{remark} Lemma \ref{lem:torsion} is the only place in our argument where we use assumption (A3), and this is only needed for Corollary \ref{cor:smallbdry}. Hence (A3) may be replaced by $|\p \Wb \sm \Wb|=0$.
\end{remark}

\section{The Lipschitz Estimate}\label{sec:lip}

We now consider a uniform version of the Lipschitz estimate on $u_{k,p}$ in Proposition \ref{prop:fp2}(2). It is unreasonable to attempt to prove a uniform Lipschitz estimate on the first $N$ eigenfunctions of $\W_p$, at least if $F$ is not strictly increasing with respect to the first eigenvalue: for at least some of the eigenvalues on which $F$ does not actually depend, we expect the corresponding eigenfunctions to not remain Lipschitz as $p\rightarrow \8$. However, we may still obtain a Lipschitz estimate on those eigenfunctions which appear in the Euler-Lagrange equation for the limiting problem. This is the aim of the estimate below.

\begin{lemma}\label{lem:lip} Let $\Wb,\W_p,F,F_p$ be as in Proposition \ref{prop:fp2}. Then there is a constant $C$ independent of $p$ such that
	\[
	\max_{x\in \W_p, k\in [1,N]} \xi_{k,p} |\n u_{k,p}(x)|^2 \leq C. 
	\]	
\end{lemma}

\begin{proof}
	Fix $k$, and consider a point $x\in \W_p \cap \{u_{k,p} = 0\}$. We aim to first show that $|\n u_{k,p}(x)|\leq C$ uniformly in $p$.
	
	Indeed, we let $u_+ = \max\{u_{k,p},0\}$ and $u_- = \max\{-u_{k,p},0\}$, and note that both of these functions satisfy $\triangle u_{\pm} \geq - C_0$ (using Lemma \ref{lem:moreest}). Let
	\[
	\Phi(r) = \1\frac{1}{r^2}\int_{B_r(x)} \frac{|\n u_+(y)|^2}{|x-y|^{n-2}} dy\2\1\frac{1}{r^2}\int_{B_r(x)} \frac{|\n u_-(y)|^2}{|x-y|^{n-2}} dy\2.
	\]
	As $u_{k,p}$ is smooth in $\W_p$, it is straightforward that
	\[
	\lim_{r\searrow 0^+} \Phi(r) = c(n)|\n u_{k,p}(x)|^2.
	\]
	On the other hand, we have from \cite{CJK}[Theorem 1.3, Remark 1.5] that
	\[
	\Phi(r) \leq C(n) (1 + \int_{B_1(x)}u_+^2 +\int_{B_1(x)}u_-^2)^2 \leq C.
	\]
	This gives the estimate promised.
	
	Now take any point $x\in \W_p\sm \{u_{k,p}=0\}$, and let $r=d(x,\{u_{k,p}=0\})$, $m=|u_{k,p}(x)|$.  Our goal is to show that $\sqrt{\xi_{k,p}}m \leq C r$, for then $\xi_{k,p}|\n u_{k,p}(x)|\leq C$ from standard elliptic estimates (using that $|\triangle u_{k,p}|\leq C_0$ on $\W_p$). Note that, fixing $r_0$ small, this follows for $r>r_0$ from Lemma \ref{lem:moreest}, so we may as well only consider $r<r_0$. Likewise, we may reduce our attention to the case of $m\geq r$, for otherwise the estimate is trivial. Using the Harnack inequality (and possibly reversing the sign of $u_{k,p}$) we have that
	\[
	\min_{B_{r/2}(x)}u_{k,p} \geq c m.
	\]
	Let $\Psi(r) = r^{2-n}$ if $n>2$ and $\Psi(r) = -\log r$ if $n=2$, and set
	\[
	v(y) = \frac{C_0}{2n}(|y-x|^2 - r^2) + \frac{c m + \frac{3 C_0}{8 n} r^2}{\Psi(\frac{r}{2})- \Psi(r)} \Psi(y-x) \qquad y\in B_r(x)\sm B_{r/2}(x).
	\]
	This function takes the values $0$ on $\p B_{r}(x)$ and $c m$ on $\p B_{r/2}$, and solves $\triangle v = C_0$. One may readily check from the formula that $v$ is positive on the annulus of definition, and has
	\[
	|\n v| \geq \frac{c m}{r},
	\]
	provided $r_0$ is chosen small enough.
	
	Applying the maximum principle to $q=u_{k,p}-v$ on $B_r(x)\sm B_{r/2}(x)$, we see that as $q\geq 0$ on the boundary and superharmonic, we have $q\geq 0$ on the interior. Take a point $z\in \p B_{r}$ with $u_{k,p}=0$; there are now two cases. If $z\in \W_p$, then we have
	\[
	C\geq |\n u_{k,p}(z)| \geq |\n v(z)|
	\]
	from our previous estimate along the nodal set, which implies the conclusion. On the other hand, if $z\in \p \W_p$, we use the viscosity solution property (V2) with $\b_i = |\n v|$ for $i=k$ and $\b_k=0$ otherwise, to learn that $\xi_{k,p} |\n v|^2 \leq \xi_0(z) \leq 2.$ This also implies the conclusion.
\end{proof}

\begin{corollary}\label{cor:ld} Let $\Wb,\W_p,F,F_p$ be as in Proposition \ref{prop:fp2}. We have that for each $x\in \bar{\W}_p$ and $r<r_*$,
	\begin{equation}\label{eq:ldc1}
	|B_r(x) \cap \W_p| \geq c_* r^n.
	\end{equation}
	Moreover, $\W_p \ss B_{R_*}$. Here $r_*,R_*,$ and $c_*$ do not depend on $p$.
\end{corollary}

\begin{proof}
	In proving \eqref{eq:ldc1}, it suffices to consider only $x$ in $\p \W_p$. Indeed, if we have that estimate on $\p \W_p$, take any $x\in \W_p$: if $B_{r/2}(x) \ss \W_p$ as well, then \eqref{eq:ldc1} holds trivially with constant $2^{-n}$; if not, then there is a $y\in \p \W_p \cap B_{r/2}(x)$, and so from the estimate on the boundary
	\[
	|B_r(x) \cap \W_p| \geq |B_{r/2}(y)\cap \W_p| \geq c_* 2^{-n}r^n
	\]
	anyway.
	
	Consider the limit
	\[
	W(x,0+) = \lim_{s\rightarrow 0}W_p(x,r),
	\]
	where $W_p(x,s)$ is the Weiss energy of Proposition \ref{prop:fp2}. If we set
	\[
	v_k^s(y)=\frac{u_{k,p}(x + sy)}{r},
	\]
	then $v_k^s$ is Lipschitz continuous uniformly in $s$. We may then find a subsequence $s_i\rightarrow 0$ along which $v_k^{s_i} \rightarrow v_k$ for some Lipschitz function $v_k$, locally uniformly and weakly in $H^1_{\text{loc}}$. From the lower bound in Proposition \ref{prop:fp2}(2),
	\[
	\max_{k,B_1} |v_k^s| \geq c(p)
	\]
	uniformly in $s$ (but not $p$), so at least one of the $v_k$ is nonzero. Take any point $z\in \{v_k \neq 0\}$: then there is a neighborhood $B_t(z)$ such that for $i$ large enough, $v^{s_i}_k \neq 0$ on $B_t(z)$. As
	\[
	\max_{B_t(z)}|\triangle v_{k}^{s_i}| \leq \l_k s_i \max_{B_{ts_i}(x+s_i z)} |u_{k,p}| \leq C s_i^2 \rightarrow 0.
	\]
	This implies that $v_k$ is harmonic on $\{v_k\neq 0\}$.
	
	 From Proposition \ref{prop:fp2}(5) and the dominated convergence theorem, we have that
	\begin{align*}
	 \int_a^b \frac{2}{t^{n+2}}\int_{\p B_t} \sum_{k=1}^N \xi_{k,p}(v_{k} - t (v_{k})_r)^2 &= \lim_i \int_{a}^{b} \frac{2}{t^{n+2}}\int_{\p B_{t}(x)} \sum_{k=1}^N \xi_{k,p}(v_{k}^{s_i} - t (v_{k})_r^{s_i})^2\\
	 & \leq \lim_i [W(x,s_i b) - W(x,s_i a) + Cs_i(b-a)]\\
	 & = 0.
	\end{align*}
	This implies that each of the $v_k$ is homogeneous of degree $1$. Consider the trace of $v_k$ on $\p B_1$ (with $k$ so that $v_k$ is nontrivial). This is a function which satisfies
	\[
	\begin{cases}
	 -\triangle_{\p B_1} v_k = \l v_k & \text{ on } \{v_k \neq 0\}\\
	 v_k = 0 & \text{ on } \p \{v_k \neq 0\},
	\end{cases}
	\]
	where $\l$ is the first Dirichlet eigenvalue of the half-sphere (it is related only to the degree of homogeneity of $v_k$). It follows from this that $|\{v_k\neq 0\}\cap \p B_1| \geq \frac{1}{2}|\p B_1|$: this is the Faber-Krahn inequality on the sphere; see \cite{CS} for details.
	
	Now, we have that for each $z\in \{v_k\neq 0\}$, $z\in \{v_k^{s_i}\neq 0\}$ for $i$ large enough, so by Fatou's lemma
	\[
	|\{v_k\neq 0\}\cap B_1| \leq \liminf_i |\{v_k^{s_i}\neq 0\}\cap B_1|\leq \liminf_i |\frac{(\W_p-x)}{s_i} \cap B_1|.
	\]
	We also have
	\[
	\int_{B_1} |\n v_k|^2 \leq \liminf_i \int_{B_1} |\n v_k^{s_i}|^2
	\]
	from the weak convergence in $H^1$, and
	\[
	\int_{\p B_1} |v_k|^2 = \lim_i \int_{\p B_1} | v_k^{s_i}|^2
	\]
	from the uniform convergence. This gives
	\begin{align*}
	\frac{|B_1|}{8} &\leq \xi_{0,p}(x)\frac{|B_1|}{2}\\
	& \leq \xi_{0,p}(x)|B_1\cap \cup_k\{v_k\neq 0\}|\\
	& = \xi_{0,p}(x)|B_1\cap \cup_k\{v_k\neq 0\}| + \int_{\p B_1}\sum_{k=1}^N \xi_{k,p} v_k((v_k)_r -  v_k)\\
	& = \xi_{0,p}(x)|B_1\cap \cup_k\{v_k\neq 0\}| + \int_{B_1} \sum_{k=1}^N \xi_{k,p} |\n v_k|^2 - \int_{\p B_1}\sum_{k=1}^N \xi_{k,p} v_k^2\\
	& \leq \liminf_i \int_{\frac{(\W_p-x)}{s_i} \cap B_1}\xi_{0,p}(x+s_i \cdot) + \int_{B_1} \sum_{k=1}^N \xi_{k,p} |\n v_k|^2 - \int_{\p B_1}\sum_{k=1}^N \xi_{k,p} v_k^2\\
	& = \liminf_i W(x,s_i)\\
	& = W(x,0+).
	\end{align*}
	The second line used that as $v_k$ is $1$-homogeneous on the domain of integration, the integrand is $0$. The third line used that $v_k$ is harmonic on $\{v_k\neq 0\}$ and integration by parts.
	
	This implies that
	\begin{align*}
	\frac{|B_1|}{8} &\leq W(x,0+) \\
	&\leq W(x,r) + Cr \\
	&\leq Cr + \frac{1}{r^n}\int_{B_r(x)\cap \W_p} \xi_{0,p} + \sum_{k=1}^N  \xi_{k,p} |\n u_{k,p}|^2\\
	&\leq Cr + C \frac{|\W_p \cap B_r|}{r^n}.
	\end{align*}
	The last line used Lemma \ref{lem:lip}. Now choose $r_*$ so that $Cr_* \leq \frac{|B_1|}{16}$ to obtain \eqref{eq:ldc1}.
	
	Let us now show that for $p$ sufficiently large, $\W_p$ is bounded uniformly in $p$. Indeed, take $R$ so that $\Wb \cc B_R$; then from Lemma \ref{lem:ex}(4) we have that $|\W_p \sm B_R|\rightarrow 0$. On the other hand, say there is a point $x$ in $\p \W_p$ with $|x| = R + r_*$. Then
	\[
	|\W_p \sm B_R| \geq |\W_p \cap B_{r_*}(x)| \geq c_* r_*^n.
	\]
	For sufficiently large $p$, this results in a contradiction, so $\W_p \ss B_{R+r_*}$.
\end{proof}

\section{The Limiting Procedure}\label{sec:lim}

We now begin the passage to the limit in $p$, which will yield an Euler-Lagrange equation for $\Wb$ and allow us to classify points on the boundary.

\begin{theorem}\label{thm:lim}
	Let $\Wb$, $\W_p$, $F$, and $F_p$ be as in Proposition \ref{prop:fp2}, and also assume that $F$ satisfies (A3-4). Then we have the following, along some subsequence:
\begin{enumerate}
 \item $\bar{\W}_p \rightarrow \Wc$ in the Hausdorff topology.
 \item $\W_p \rightarrow \W_\8\ss \Wb$ in the weak $\g$ sense.
 \item $\p \W_p \rightarrow Z$ in the Hausdorff topology, where $\p \Wb \ss Z \ss \Wc$.
 \item There is an $r_*>0$ and a constant $c_*$, both depending only on $n$, $F$, and $\Wb$, such that $|B_r(x)\cap \W_\8|> c_* r^n$ for every $x\in \Wc$ and $r<r_*$.
 \item Let $\cK$ be the set of $k$ for which $\xi_k>0$. Then $\l_k(\W_p) \rightarrow q_k$, and $q_k = \l_{\max \{j \in \cK: q_j = q_k \}}(\Wb)$.
 \item There is an orthonormal set $\{u_k\}_{k\in \cK}$ of Lipschitz eigenfunctions of $\Wb\sm Z$, with eigenvalues $q_k$, for which if we define the Weiss energies
	 \[
	 W(x,r) = \frac{1}{r^n} \int_{B_r(x)\cap \Wb} (\sum_{k\in \cK} \xi_{k}|\n u_{k}|^2 + 1) - \frac{1}{r^{n+1}}\int_{\p B_r(x)} \sum_{k\in \cK}\xi_{k}u_{k}^2 d\cH^{n-1},
	 \]
	 then they satisfy, for $s<r$ and $x\in Z$,
	\[
	W(x,r) - W(x,s) \geq \int_s^r \frac{2}{t^{n+2}}\int_{\p B_t(x)} \sum_{k\in \cK} \xi_{k}(u_{k} - t (u_{k})_r)^2 - C(r-s).
	\]
\end{enumerate}
\end{theorem}

\begin{proof} 
	We have that the $\W_p$ are uniformly bounded open sets; it follows that there exists a quasiopen set $\W_\8$ such that, along a subsequence of $p$, $\W_p$ converge to $\W_\8$ in the weak $\g$ sense. In particular, we have that
	\[
	\l_k(\W_\8) \leq \liminf_p \l_{k}(\W_p)
	\] 
	for every $k$, and $|\W_\8|\leq \liminf_p |\W_p|$. It follows that
	\[
	F(\W_\8) + |\W_\8| + E(\W_\8) \leq \liminf_p F(\W_p) + |\W_p| + E(\W_p) = F(\Wb) + |\Wb|,
	\]
	with right equality from Lemma \ref{lem:ex}(4). It follows that $E(\W_\8)=0$ (implying $|\W_\8| = |\Wb|$ and $\text{Int}(\Wb) \ss \W_\8 \ss \Wc$ a.e.), and $\W_\8$ is an $F$-minimizer. From this we also have that 
	\begin{equation}\label{eq:limi1}
	F(\W_\8)=F(\Wb)=\lim_{p\rightarrow \8} F_p(\W_p).
	\end{equation}
	
	From Corollary \ref{cor:smallbdry}, we know that $|\Wc\sm \Wb|=0$, so $|\W_\8 \sm \Wb|=0$. Using (B3) and the fact that $\W_\8$ is an $F$-minimizer, this implies that $\text{cap}(\W_\8\sm \Wb)=0$, so up to choosing a different representative for $\W_\8$ this gives $\W_\8 \ss \Wb$, establishing (2). In particular, we have that $\l_k(\W_\8) \geq \l_k(\Wb)$ for every $k$.

	Let us consider the values $\xi_k = \lim_p \xi_{k,p}$ more carefully in light of this information. Up to passing to a further subsequence, we may assume that $\l_k(\Wb)\leq \l_k(\W_\8)\leq q_p = \lim_p \l_k(\W_p)$.  From \eqref{eq:limi1}, we have that
	\begin{equation}\label{eq:limi2}
	F(q_1,\ldots, q_k, \ldots, q_N) = F(q_1,\ldots, \l_k(\Wb), \ldots, q_N)
	\end{equation}
	for every $k$. From \eqref{eq:fpdir}, we have that
	\begin{equation}\label{eq:limi4}
	\xi_{k,p} = \p_{\k_k} F_p(\W_p) =  o_p(1)  +   \sum_{j=k}^N \1 \frac{\l_k(\W_p)}{\t_{k,p}(\W_p)} \2^{p-1} \cdot \p_{\k_j} G_p(\W_p).
	\end{equation}
	Set $\m_{k,p} = \p_{\k_k} G_p(\W_p)$; from Lemma \ref{lem:moreest}(1) we have that the $\m_{k,p}$ are uniformly bounded in $p$, and so may pass to a further subsequence along which $\m_{k,p}\rightarrow \m_k$ for some numbers $m_k$.
	
	If $\l_k(\Wb) < q_k$, then we may combine \eqref{eq:limi2} with (A4) in the following way:
	\begin{align*}
	&0 = \lim_{p\rightarrow \8} \1\frac{p}{2}\2^{N-1}\int_{[0,\frac{2}{p}]^{N-1}} |F(q_1 +\z_1,\ldots,q_k + h, \ldots) - F(q_1 +\z_1,\ldots,q_k, \ldots)| d\z_1\ldots d\z_N \\
	  &\geq \frac{1}{2^{N-1}} \limsup_{p\rightarrow \8} \int_{[0,\frac{2}{p}]^{N-1}} |F(\l_1(\W_p) +\z_1,\ldots,\l_k(\W_p) + h,\ldots) - F(\l_1(\W_p) +\z_1,\ldots,\l_k(\W_p),\ldots)| d\z_1\ldots d\z_N &\\
	  & = \frac{1}{2^{N-1}} \limsup_{p\rightarrow \8} \p_{\k_k}G_p(\W_p)\\
	  & = \frac{1}{2^{N-1}} \lim_{p\rightarrow \8} \m_{k,p}\\
	  & = \frac{1}{2^{N-1}} \m_k.
	\end{align*}
	The first line used (A4), the second line used the definition of $q_k$, and the third line used the definition of $G_p$, in \eqref{eq:gp}. If instead we have $\l_k(\Wb) = \l_{k+1}(\Wb)$, then from (B2) we know that $F$ (and hence $G$) do not depend on $\k_k$. This means that $\m_{k,p}=0$, and so $\m_k=0$. In other words, we have shown that
	\begin{equation}\label{eq:limi3}
	\m_k > 0 \quad \Rightarrow \quad \l_k(\Wb) = q_k < \l_{k+1}(\Wb).
	\end{equation}

	Passing to another subsequence, we may take the $u_{k,p}$ (for all $k\leq N$) to converge weakly in $H^1$ and strongly in $L^2$ to functions $v_k$ in $H^1_0(\W_\8)$ (see Lemma 5.3.5 in \cite{DB}). In particular, we have that the $v_k$ are orthonormal in $L^2$, and that
	\[
	\s_k := R[v_k]\leq \liminf R[u_{k,p}] = \lim \l_k(\W_p) = q_k.
	\]
	Let $\{\a_j\}_{j=1}^J$ be an increasing sequence of natural numbers with the property that $\m_k>0$ if and only if $k$ is one of the $\a_j$. From \eqref{eq:limi3}, this implies that $\l_{\a_j}(\Wb)<\l_{\a_j + 1}(\Wb)$. For each $j$, let $\cE_j$ to be the set of $k$ such that $\s_k = \s_{\a_j}$; these are sets of several consecutive integers of which the largest is $\a_j$. Let $\cK$ be the set of $k$ for which $\xi_k\neq 0$.

	By using the $v_k$ as test functions in the min-max formula for the definition of $\l_k(\W_\8)$, we see that $\l_k(\W_\8) \leq \s_k$. For every $k$ for which $q_k = \l_k(\Wb)$, we also have that $\s_k = \l_k(\W_\8) = q_k$, and so for such a $k$ the convergence is strong in $H^1$. In particular, this holds for $k=\a_j$. Now take any $k\in \cE_j$ for some $j$:
	\[
	\s_{\a_j} = \s_k \leq q_k \leq q_{\a_j} = \s_{\a_j}.
	\]
	Hence for any $k\in \cup_j \cE_j$, these inequalities are equalities, and we have $q_k=\l_{\a_j}(\Wb)$.
	
	Let us now show that $\cK \ss \cup_{j=1}^J \cE_j$. This fact comes from inspecting the formula \eqref{eq:limi4}. Indeed, take a $k$ which does not lie in any of the sets $\cE_j$: then for every $l\geq k$, we have that either $\lim_{p}\m_{l,p} = \m_l = 0$ or $l=\a_j$ but $ q_k < q_l$. For every $l$ with $q_k<q_l$, we have that
	\[
	\lim_{p\rightarrow \8} \1\frac{\l_k(\W_p)}{\t_{l,p}(\W_p)}\2^{p-1} \leq \lim_{p\rightarrow \8} \1\frac{\l_k(\W_p)}{\l_l(\W_p)}\2^{p-1} = 0,
	\]
	so in either case the corresponding term in \eqref{eq:limi4} goes to $0$ as $p\rightarrow \8$. It follows that $\xi_k=0$, so $k\notin \cK$. 
	
	On the other hand, a similar argument shows that $\xi_{\a_j}>0$. Indeed, the first term in the sum in \eqref{eq:limi4} has
	\[
	\liminf_p \1\frac{\l_{\a_j}(\W_p)}{\t_{\a_j,p}(\W_p)} \2^{p-1} \cdot \m_{k,p} \geq \frac{\m_{\a_j}}{\dim \cE_j}>0.
	\] 
	This implies (5): we have that for every $k\in \cK$, there is a $j$ with $q_k = q_{\a_j} = \l_{\a_j}(\Wb)$, and this $\a_j$ is characterized by $\a_j = \max\{i\in \cK: q_i = q_k\}$.
	
	For each $k\in \cK$, we have that $u_{k,p}$ is  Lipschitz uniformly in $p$, and (up to further subsequences) we have $u_{k,p}\rightarrow v_k$ uniformly. Up to passing to a further subsequence, we may assume that $\bar{\W}_p \rightarrow K$ and $\p \W_p \rightarrow Z \ss K$ in Hausdorff topology, for some closed, bounded, sets $K$ and $Z$.
	
	We now turn our attention to (1), (3) and (4). We first show that $K = \Wc$. Indeed, take any point $x \in K$, and fix $r<r_*$ smaller than the constant of Corollary \ref{cor:ld}; then for all $p$ large enough we have
	\[
	|B_r(x)\cap \W_p| \geq c_* r^n.
	\]
	If $x\notin \Wc$, then for a small enough $r$ and all $p$ large we have $\W_8 \ss \Wb \cc \R^n \sm B_r(x)$, and so (from the weak $\g$ convergence)
	\[
	|\W_\8| \leq \liminf_{p\rightarrow \8} | \W_p \sm B_r(x)| \leq \lim_{p\rightarrow \8} | \W_p | - c_* r^n \leq |\W_\8| - c_* r^n.
	\] 
	This is a contradiction. On the other hand, taking any $x\in \Wb \sm K$, there is a ball $B_r(x)$ with $K\cap B_r(x) = \emptyset$ and $|\Wb \cap B_r(x)| = c >0$ (the latter follows as any quasiopen set contains only points of Lebesgue density $1$ for itself). For large $p$, then, $\W_p \cap B_r(x) = \emptyset$, meaning that $\W_\8 \ss \Wb \sm B_r(x)$. This contradicts that $|\W_\8| = |\Wb|$. We have now shown (1).
	
	As a consequence, $1_{\W_p} \rightarrow 1_{\Wc}$ strongly in $L^1$. To see this, first note that on $\R^n \sm \Wc$, we have $1_{\W_p}\rightarrow 0$ from the Hausdorff convergence, so
	\[
	|\W_p \sm \Wc| \rightarrow 0
	\] 
	by the dominated convergence theorem. On the other hand, this gives
	\[
	\lim_{p} |\W_p \cap \Wc| = \lim_p |\W_p | - \lim_p |\W_p \sm \Wc| = |\W_\8| - 0 = |\Wc|,
	\]
	where the middle step used that $|\Wc\sm \Wb|=0$ and $|\Wb \sm \W_\8|=0$. Applying this and using the estimate in Corollary \ref{cor:ld} then gives (4).
	
	Take a point $x\in \p \Wb$. If for some ball $B_r(x)$, we have $|B_r(x)\sm \Wb|=0$, then we may apply property (B3) to the set $\W = \Wb \cup B_r(x)$ (this has $\l_k(\W)\leq \l_k(\Wb)$, as $\Wb \ss \W$, and $|\Wb|=|\W|$ by definition, so $\W$ is also an $F$-minimizer). This means that $\text{cap}(B_r(x)\sm \Wb)=0$, and so, by our standard choice of representative for $\Wb$, $B_r(x)\ss \Wb$: this contradicts $x\in \p \Wb$. On the other hand, we have from the just-proved (4) that $|\Wb \cap B_r(x)|>0$ for every $r>0$. Thus for every $r>0$, 
	\begin{equation}\label{eq:limi5}
	\frac{|B_r(x)\sm \Wb|}{|B_r|}\in (0,1).
	\end{equation}
	Using that $1_{\W_p} \rightarrow 1_{\Wb}$ in $L^1$,
	\[
	\lim_{p\rightarrow \8}\frac{|B_r(x)\sm \W_p|}{|B_r|} = \frac{|B_r(x)\sm \Wb|}{|B_r|} \in (0,1) 
	\]
	for each fixed $r$.	We may then find a sequence $r_p\rightarrow 0$ for which
	\[
	\frac{|B_{r_p}(x)\sm \W_p|}{|B_{r_p}|} \in (0,1),
	\]
	and then take points $x_p \in B_{r_p}\cap \p \W_p$. By construction, these points $x_p \rightarrow x$ as $p\rightarrow \8$, showing that $\p \Wb \ss Z$ (recall that $Z$ is the Hausdorff limit of $\p \W_p$). This implies (3).
	
	We may now easily verify the first part of (6). Indeed, take any $x$ in the open set $\Wb \sm Z$; then there is a ball $B_r(x)\cc \Wb \sm Z$, so $B_r(x) \ss \W_p$ for all $p$ large. On this ball, the functions $u_{k,p}$ have uniformly bounded $C^M$ norms for any $M$, and so converge in $C^M$ topology to $v_k$. It follows that $-\triangle v_k = q_k v_k$ for every $k\in \cK$ on $\Wb \sm Z$, while $v_k = 0$ on $Z$, which contains $\p (\Wb \sm Z)$. This implies that $v_k$ is an eigenfunction of $\Wb \sm Z$. We may augment the $\{v_k\}$ to a basis $\{u_k\}$ of eigenfunctions of $\Wb \sm Z$ (we preserve the numbering so that $u_k = v_k$ for $k\in \cK$).

	Let us now consider the Weiss energies $W_p$. Let $\xi_0(x)$ be the uniform limit of $\xi_{0,p}(x)$ (this has the explicit formula $\xi_{0,p}(x) = 1 + s \max \{d(x,\Wb),1\} - s \max \{d(x,\Wb^c),1\}$, and so equals $1$ on $\p \Wb$). Take any point $x\in Z$, and fix a sequence $x_p\rightarrow x$ with $x_p\in \p \W_p$: we claim that for every $r>0$, $W_p(x_p,r)\rightarrow W(x,r)$. Indeed, we know the following: the functions $u_{k,p}$ are bounded (uniformly in $p$ for all $k\leq N$, from Lemma \ref{lem:moreest}(3)), so for each $k\notin \cK$, we have
	\[
	\lim_p \xi_{k,p} \int_{\p B_r(x_p)} u_{k,p}^2 = 0.
	\]
	We also have, from weak convergence in $H^1$, that for each of those $k$
	\[
	\lim_p \xi_{k,p} \int_{B_r(x_p)} |\n u_{k,p}|^2 = 0.
	\]	
	On the other hand, for $k\in \cK$, we have that $u_{k,p}\rightarrow u_k$ uniformly and strongly in $H^1$, so that
	\[
	\lim_p \xi_{k,p} \int_{\p B_r(x_p)} u_{k,p}^2 = \xi_{k} \int_{\p B_r(x)} u_{k}^2
	\]
	and
	\[
	\lim_p \xi_{k,p} \int_{ B_r(x_p)} |\n u_{k,p}|^2 = \xi_{k} \int_{ B_r(x)} |\n u_{k}|^2.
	\]
	These together imply our claim. A similar argument now shows that
	\[
	W(x,r) - W(x,s) \geq \int_s^r \frac{2}{t^{n+2}}\int_{\p B_t(x)} \sum_{k=1}^N \xi_{k}(u_{k} - t (u_{k})_r)^2 - C(r-s)
	\]
	for all $s<r$, by taking the limit of the right-hand side of the estimate in Proposition \ref{prop:fp2}(5).
\end{proof}

\begin{remark} The nature of the numbers $\xi_k$ and $\m_k$ and the set $\cK$ in this theorem may appear mysterious, even though they play an essential role in understanding the problem. In the case of $F$ being locally $C^1$, however, it is not difficult to verify that
	\[
	\m_k = \p_{\k_k} F(\Wb).
	\] 
	For $k = \a_j$, we are guaranteed that
	\[
	\xi_k \in [ \frac{\m_{\a_j}}{\dim \cE_j}, \m_{\a_j} ],
	\]
	but for the other $k\in \cE_j$ we only know that $\xi_k \in [0,\xi_{\a_j}]$. One may also check that
	\[
	\sum_{k\in \cE_j} \xi_k = \m_{\a_j}.
	\]
	Understanding the precise nature of the $\xi_k$ is an interesting problem we leave open. In particular, though, knowing that
	\[
	\min_{k\in \cE_j} \xi_k > 0
	\] 
	would be of great interest for understanding the nature of the set $Z_C$ introduced in Section \ref{sec:reg}; a particularly tempting conjecture might be that $\xi_k = \frac{\m_{\a_j}}{\dim \cE_j}$.
\end{remark}

\begin{remark}\label{rem:smallermin} We do not show here that the $u_{k}$ are eigenfunctions of $\Wb$. It is perhaps plausible that they are in fact not: let us suggest the following conceptual exercise to clarify the situation. Suppose that $F = \l_N$ and there is a minimizer $\Wb$ which has the properties that (a) $\Wb$ is open, (b) $\l_N(\Wb)$ is simple, and (c) the nodal set $\{u_*=0\}\cap \Wb$ of the eigenfunction $u_*$ associated to $\l_N(\Wb)$ contacts the boundary $\p \Wb$ (i.e., it is not compactly contained in the interior of $\Wb$). Then take a point $x \in  \p\Wb$ which is in the closure of the nodal set, and let $\W_t = \Wb \sm (B_t(x)\cap \{u_*=0\})$. The mapping $t\mapsto \l_k(\W_t)$ is continuous, and preserves the fact that $u_*$ is an eigenfunction of $\W_t$ (hence for small $t$, $\l_{N-1}(\W_t)<\l_N(\W_t) = \l_N(\Wb)$. As $t$ increases, so do the other eigenvalues of $\W_t$, so for a critical $t_*$, we will have $\l_{N-1}(\W_{t_*}) = \l_N(\W_{t_*})= \l_N(\Wb)$. For this domain $\W_{t_*}$, there is then another eigenfunction $u_{**}$ with eigenvalue $\l_{N}(\Wb)$. It is possible that it is this function $u_{**}$ which numbers among the $u_k$ of Theorem \ref{thm:lim}(6), rather than (or in addition to) $u_*$.
	
An important observation about this $\W_{t_*}$ is that it is also a minimizer of $F=\l_N$ (although we are unable to prove the same for the set $\Wb\sm Z$ in Theorem \ref{thm:lim}, this is likely a technical matter). Therefore, the issue was that the original $\Wb$ was just not the ``correct'' representative, in the sense that it was not the representative whose eigenfunctions are used to form the first variation formula. We use (B3) as a selection criterion for representatives, which will choose the largest one; this is consistent with other approaches in the literature \cite{Bucur,BMPV}. Selecting smaller representatives may be needed to more fully understand the structure of $\p \Wb$.

Let us finally note that this example is contrived. It is conjectured that (b) never holds (see \cite{H}), and if $\l_N$ is not simple this construction is impossible unless all of the eigenfunctions with eigenvalue $\l_N(\Wb)$ have a piece of their nodal set in common.
\end{remark}

\section{Limits of Viscosity Solutions} \label{sec:info}

We start with a standard lemma, which shows that appropriate limits of viscosity solutions are viscosity solutions themselves.

\begin{lemma}\label{lem:visc} Let $\{v_{k,p}\}_{k=1}^N$, $V_p$ constitute a viscosity solution with parameters $\xi_{k,p}$ and $C_*$ on $U$. Assume that as $p\rightarrow \8$, $\xi_{k,p}\rightarrow \xi_k$ (uniformly for $\xi_{0,p}$). Also assume that $\xi_{k,p} v_{k,p}$ are uniformly Lipschitz continuous and converge to functions $\xi_k v_k$ uniformly on $U$. Finally, assume that there is an open set $V\ss U$ such that $\p V_p \rightarrow \p V$ and $\bar{V}_p \rightarrow \bar{V}$ locally in the Hausdorff sense on $U$. Then $\{v_{k,p}\}_{\{k:\xi_k>0\}}$ and $V$ constitute a viscosity solution on $U$ with parameters $\xi_k$ and $C_*$.
\end{lemma}

\begin{proof}
	First, the property (V1) is straightforward: for any ball $B_r(x)\cc V$, we have $B_r(x)\ss V_p$ for all $p$ large, and so as distributions $-\triangle v_{k,p} \rightarrow -\triangle v_{k}$ on $B_r(x)$. It follows that $|-\triangle v_k|\leq C_*$.

	Take any ball $B_{r_0}(y_0)\ss V$ and numbers $\b_k$ as in (V2):
	\[
	|v_{k}(z)| > \b_k e\cdot (z-x)) - |z-x|\w(|z-x|)
	\]
	on $B_{r_0}(x)$, with $\w$ a continuous increasing function with $\w(0)=0$ and $e = \frac{y_0-x}{|y_0-x|}$. Note that we may assume that at least one of the $\b_k$ is nonzero, for otherwise there is nothing to check.  Let $r = s r_0$ and $y = (1-s)x + s y$ for some $s$ very small; then we have that $\{x\} = \p B_{r}(y) \cap \p V$ and
	\[
	|v_{k}(z)| > \b_k (e\cdot (z-x)) - |z-x|\w(|z-x|)
	\]	
	on $B_{r_0}(x)$. By choosing $s$ small enough, we may arrange so that $\w(4r)\leq  \e \min_{\{k:\b_k> 0\}} \b_k$. 
	
	We have that for for all $p$ large enough and $k$ with $\xi_k$ nonzero,
	\begin{align}
	|v_{k,p}(z)| &> |v_{k}(z)| - c_p \nonumber\\
	&> \b_k (e\cdot (z-x)) - |z-x|\w(|z-x|) - c_p \nonumber\\
	&> \b_k (e\cdot (z-x)-2\e r). \label{eq:visci1}
	\end{align}
	on $B_{2r}(x)$, with $c_p\rightarrow 0$. On the other hand, we also have that for $p$ large enough, $d(x,\p V_p) < \e r$.
	
	Define $q =x -4\e e$,  $R_T = \sqrt{r^2 + T^2}$, $x_T = x + (4\e r + T)e$, and, for $T\geq 4$
	\[
	w_T(z) = \begin{cases}
				\frac{(|z-x_T|^{2-n} - R_T^{2-n})^+}{(n-2) R_T^{1-n}} & T<\8\\
				 (e \cdot(z-q))^+  & T= \8
			\end{cases}
	\]
	(and similarly with the logarithm when $n=2$). Let $P_T = \{w_T>0\}$; this is either a large ball or a half-space. Notice that $w_T$ is harmonic on $P_T \sm \{x_t\}$, so in particular on $P_T \cap B_{2r}(x)$. Moreover, $P_T \cap \{z:e\cdot (z-q)=0  \} = B_r(q)\cap \{z:e\cdot (z-q)=0  \}$ for every $T<\8$. Finally, we have that $w_T \rightarrow w_\8$  as $T\rightarrow \8$, locally uniformly, and $\n w_T \rightarrow e$ uniformly on $B_{2r}(x)\cap P_T$.
	
	Set $C_0 = \frac{C_*+1}{\min_{\{k:\b_k> 0\}} \b_k}$, and let $\z_T$ be the solution of
	\[
	\begin{cases}
		-\triangle \z_T = - C_0 & \text{ on } P_T \cap B_{2r}(x)\\
		\z_T = 0 & \text{ on } \p(P_T \cap B_{2r}(x)).
	\end{cases}
	\]
	We have that $0\leq - \z_T(z) \leq C d(z,P_T)$ and $|\n \z_T| \leq C r <\frac{1}{2}$ provided that $r$ is chosen sufficiently small. We extend $\z_T$ by $0$ to the remainder of $\R^n$, and set $\phi_T = w_T + \z_T$. Importantly, observe that $\{\phi_T > 0\} = P_T$, by our estimates on $\z_T$.
	
	Now, observe that from \eqref{eq:visci1}, we have that $\b_k \phi_\8 < |v_{k,p}|$ on $P_\8 \cap B_{2r}(x)$, and so $\b_k \phi_T < |v_{k,p}|$ on $P_T \cap B_{2r}(x)$ for each $T$ large enough. Let $T_*$ be chosen by
	\[
	T_* = \inf\{ T\geq 4: \b_k \phi_T < |v_{k,p}| \text{ on } P_T \cap B_{2r}(x) \forall k \text{ with } \b_k,\xi_k>0 \}.
	\]
	We have shown that $T_*<\8$.  Recall that $d(x,\p V_p) < \e r$: we combine this with the fact that $B_\e(x) \ss P_T$ for all $T$ with $R_T - T > 5\e r $ (so, for all $T$ with $T<\frac{1}{5\e}$), which means that for any such $T$, at some point in $B_\e(x)$ all of the $v_{k,p}=0$ while $\phi_T>0$. In particular, this means that $T_* \geq  \frac{1}{5\e}$, and $T_*\rightarrow \8$ as $\e\rightarrow 0$. We also know that $|v_{k,p}(z)| \geq \b_k \phi_{T_*}(z)$.
	
	Now, there are two possibilities for what occurs at $T_*$: either $|v_{k,p}(z)| = \b_k \phi_{T_*}(z)>0$ for some $z$ and $k$ with $\b_k>0$, or there is a $z\in \p V_p \cap \p (P_{T_*} \cap B_{2r}(x))$. The first case is actually impossible. Indeed, at such a $z$ we would have that
	\[
	 \b_k \triangle \phi_{T_*} \leq C_*,
	\]
	as $\b_k \phi_{T_*}$ touches either $v_{k,p}$ or $-v_{k,p}$ from below at $z$, and this function has Laplacian controlled by $C_*$. But by construction, $\b_k \triangle \phi_{T_*}(z) = \b_k C_0 > C_*$, which is a contradiction. 
	
	It follows that we are in the second case: there is a point $z\in \p V_p \cap \p (P_{T_*} \cap B_{2r}(x))$. First of all, notice that this point must actually be in $\p P_{T_*} \cap B_{2r}(x)$. This is because the sets $P_{T} \cap \p B_{2r}(x)$ are decreasing as $T$ decreases, so if $z$ was in $P_{T_*} \cap \p B_{2r}(x)$, it would also be in $P_{T_\8} \cap \p B_{2r}(x)$, which is impossible by construction.
	
	We may therefore use the smoothness of $P_{T_*}$ to apply the property (V2) of $\{v_{k,p}\}$ to this point $z$, with constants $\b_k' = \b_k |\n \phi_{T_*}(z)|$ if $\b_k$ and $\xi_k$ are nonzero, and $\b_k'=0$ otherwise. This tells us that
	\[
	\sum_{k} (\b_k')^2 \xi_{k,p} \leq \xi_{0,p}(z).
	\]
	Taking the limit as $p\rightarrow \8$, this gives
	\[
	\sum_{k} (\b_k')^2 \xi_{k} \leq \xi_{0}(z) \leq \xi_0(x) + Cr.
	\]
	Moreover, we have that $\b_k'\rightarrow \b_k$ as $T_* \rightarrow \8$. Sending $r\rightarrow 0$ (and hence $\e\rightarrow 0$ and $T_*\rightarrow \8$) gives
	\[
	\sum_{k} \b_k^2 \xi_{k} \leq \xi_0(x),
	\]
	which is the property (V2).
	
	The property (V3) follows in a similar manner.
\end{proof}

\begin{lemma} \label{lem:weissfacts}Let $F, \Wb$ satisfy (A1-4) and (B1-5). Then:
	\begin{enumerate}
		\item For each $x\in Z$,
		\begin{equation}\label{eq:wfc1}
		W(x,0+) = \lim_{r\searrow 0} W(x,r) = \lim_{r\searrow 0} \frac{|B_r(x)\cap \Wb|}{r^n}.
		\end{equation}
		In particular, the limit on the right exits.
		\item $0<c_0 \leq W(x,r)\leq C_0$ for all $x\in Z$ and $r<r_0$.
		\item $|W(x,r) - W(y,r)|\leq C \frac{|x-y|}{r}$ for all $x,y\in Z$ and $r>0$.
	\end{enumerate}
\end{lemma}

\begin{proof}
	For the first property, we may take $x=0$ without loss of generality. Take any sequence $r_i\searrow 0$; we will show that along some subsequence of this $r_{i_l}$,
	\[
	\lim_l \frac{1}{r^n_{i_l}} \int_{B_{r_{i_l}}\cap \Wb} (\sum_{k\in \cK} \xi_{k}|\n u_{k}|^2) - \frac{1}{r^{n+1}_{i_l}}\int_{\p B_{r_{i_l}}} \sum_{k\in \cK}\xi_{k}u_{k}^2 d\cH^{n-1} = 0.
	\]
	This implies that
	\[
	W(0,0+) = \lim_{l\rightarrow \8} W(0,r_{i_l}) = \lim_{l\rightarrow \8} \frac{|B_{r_{i_l}}\cap \Wb|}{r_{i_l}^n} = \lim_{r\searrow 0} \frac{|B_r(x)\cap \Wb|}{r^n},
	\]
	from which \eqref{eq:wfc1} follows. We abbreviate $r_{i_l}=r$ below, passing to further subsequences as needed without changing notation.
	
	Define functions
	\[
	v^r_k(x) = \frac{u_k(r x)}{r}.
	\]
	These are uniformly Lipschitz continuous, so we may take a subsequence such that $v_k^{r} \rightarrow v_k$ for some Lipschitz functions $v_k$ uniformly and weakly in $H^1_{\text{loc}}$. We may also check that $-\triangle v_k=0$ and $v_k^{r} \rightarrow v_k$ locally in $C^m$ topology for any $m$ on the set $\{v_k \neq 0\} $: indeed, for at any point $x$ in this set, there is a ball $B_r(x)$ on which $v_k^r$ is nonzero for all $r$, and so satisfies $-\triangle v_k^r = q_k r^2 v_k^r$. From interior estimates this means $v_k^r$ is uniformly bounded in $C^m$, which implies the convergence and that $v_k$ is harmonic.
	
	We also have that $\n v_k^{r} \rightarrow \n v_k$ strongly in $L^2(B_R)$ for every $R$. Indeed, we have that
	\[
	\triangle |v_k^r| \geq - C r^2 \geq -C
	\]
	in the distributional sense on $\R^n$, and so may be represented by a Borel measure. On the other hand,
	\[
	\triangle |v_k^r| (B_R) \leq \int_{\p B_R} |\n v_k^r| \leq CR^{n-1},
	\] 
	so these measures are uniformly bounded in $r$. It follows from the compact embedding (see \cite{P}[Proposition 5.10]) that $|v_k^r|$ converges strongly to $|v_k|$ in $W^{1,p}(B_R)$ for any $p<\frac{n}{n-1}$. As we also have $|\n v_k|$ is uniformly bounded, this gives that $\n v_k^r \rightarrow \n v_k$ strongly in $L^p(B_R)$ for all $p<\8$, so in particular for $p=2$.
	
%	First, on $\{v_k \neq 0\}\cap B_R$, this is immediate from the dominated convergence theorem. After passing to a subsequence, we may find a Radon measure $\s$ so that $|\n v_k|^2 d\cL^n \rightharpoonup \s$ locally in the weak-* sense.	The fact that $\n v_k^r$ is uniformly bounded guarantees that $\s$ is absolutely continuous with respect to Lebesgue measure. From standard properties of Lipschitz functions, we have that $\n v_k = 0$ a.e. on $\{v_k=0\}$. Take any $x$ with $v_k(x)=0$ and
%	\begin{equation}\label{eq:wfi2}
%	\w(x,s) := \max_{B_s(x)} \frac{|v(y)|}{s} \quad \text{ has } \quad  \lim_{s\searrow 0}\w(x,s) = 0.
%	\end{equation}
%	As $- \triangle |v_{k}^r| \leq C r^2$, we may apply the Caccioppoli inequality to $|v_k^r|$ to see that
%	\[
%	\int_{B_s(x)} |\n v_k^r|^2 \leq \frac{C}{s^2} \int_{B_{2s}(x)} (v_k^r)^2 + C s^n r^2.
%	\] 
%	The right-hand side converges to
%	\[
%	\frac{C}{s^2} \int_{B_{2s}(x)} (v_k)^2 \leq Cs^n \max_{B_s(x)} \frac{|v(y)|^2}{s^2} \leq C s^n \w^2(x,s),
%	\]
%	with the last step from \eqref{eq:wfi2}. Thus
%	\[
%	\frac{\s(B_s(x))}{|B_s|} \leq C\w^2(x,s) \rightarrow 0
%	\]
%	as $s\searrow 0$. As this is true at almost every $x\in \{v_k = 0\}$,  the Radon-Nykodym theorem implies that $\s(\{v_k=0\})=0$, which gives that
%	\[
%	0 = \s(B_R \cap \{v_k = 0\}) \geq \limsup_{r\rightarrow 0} \int_{\{v_k=0\}\cap B_R}|\n v_k|^2.
%	\]
%	This implies the strong convergence.
	
	We then have that
	\begin{align}
	\lim_{r\searrow 0}& \frac{1}{r^n} \int_{B_{r}\cap \Wb} (\sum_{k\in \cK} \xi_{k}|\n u_{k}|^2) - \frac{1}{r^{n+1}}\int_{\p B_{r}} \sum_{k\in \cK} \xi_{k}u_{k}^2 d\cH^{n-1}  \nonumber\\
	&=\lim_{r\searrow 0} \int_{B_1} (\sum_{k\in \cK} \xi_{k}|\n v_{k}^{r}|^2) - \int_{\p B_1} \sum_{k\in \cK} \xi_{k}(v_{k}^{r})^2 d\cH^{n-1}\nonumber\\
	&= \int_{B_1} (\sum_{k\in \cK} \xi_{k}|\n v_{k}|^2) - \int_{\p B_1} \sum_{k\in \cK} \xi_{k}v_{k}^2 d\cH^{n-1}.\label{eq:wfi1}
	\end{align}
	From Theorem \ref{thm:lim}(6), we have that
	\begin{align*}
	 0 & = \lim W(x,b r) - W(x,a r) \\
	 &\geq \lim \int_a^b \frac{2}{t^{n+2}}\int_{\p B_t(x)} \sum_{k\in \cK} \xi_{k}(v_{k}^{r} - t (v_{k}^{r})_r)^2 \\
	 &= \int_a^b \frac{2}{t^{n+2}}\int_{\p B_t(x)} \sum_{k\in \cK} \xi_{k}(v_{k} - t (v_{k})_r)^2.
	\end{align*}
	This, for all $a<b$, implies that the $v_k$ are $1$-homogeneous. Together with the fact that the $v_k$ are harmonic when nonzero, integrating by parts gives that the last line in \eqref{eq:wfi1} is $0$. This implies (1).

	For (2), the lower bound is a direct consequence of (1) and Theorem \ref{thm:lim}(4), while the upper bound is trivial from the Lipschitz bound on $u_k$. For (3), observe that we may, without loss of generality, assume that $|x-y|\leq \e r$ for a small $\e$, as otherwise this follows from the upper bound in (2). Then using that $|B_r(x)\triangle B_r(y)| \leq C(n) |x-y| r^{n-1}$,
	\begin{align*}
	|\frac{1}{r^n}\int_{B_{r}(x)\cap \Wb} (1+\sum_{k\in \cK} \xi_{k}|\n u_{k}|^2) - \frac{1}{r^n}\int_{B_{r}(y)\cap \Wb} (1+\sum_{k\in \cK} \xi_{k}|\n u_{k}|^2)| 
	&\leq C  \frac{|B_r(x)\triangle B_r(y)|}{r^n} \\
	&\leq C \frac{|x-y|}{r}.
	\end{align*}
	For the other term,
	\begin{align*}
	|\frac{1}{r^{n+1}}&\int_{\p B_{r}(x)\cap \Wb} \sum_{k\in \cK} \xi_{k} u_{k}^2 - \frac{1}{r^{n+1}}\int_{\p B_{r}(y)\cap \Wb} \sum_{k\in \cK} \xi_{k} u_{k}^2| \\
	&\leq C r  |\frac{1}{r^{n+1}}\int_{\p B_{r}(x)\cap \Wb} \sum_{k\in \cK} \xi_{k} |u_{k}(z) - u_k(z+(y-x))|d\cH^{n-1}(z) | \\
	&\leq C r \frac{|x-y| r^{n-1}}{r^{n+1}}\\
	&\leq C \frac{|x-y|}{r}.
	\end{align*}	
	Combining gives (3).
\end{proof}

\begin{lemma}\label{lem:viscsol} Let $F, \Wb$ satisfy (A1-4) and (B1-5) and $\theta < 1$. Let $x\in Z$ and $r\in (0,r_0(\theta))$ such that $W(x,r) \leq \theta |B_1|$. Then there is a $\g = \g(\theta)$ such that $\Wb$ is relatively open in $B_{\g r}(x)$, $Z\cap B_{\g r}(x) = \p \Wb \cap B_{\g r}(x)$ and $\{u_k\}_{k\in \cK}$ and $B_{\g r}(x)\cap \Wb$ constitute a viscosity solution with parameters $\{\xi_k\}$ and $C$.
\end{lemma}

\begin{proof} Without loss of generality, take $x=0$. Choose $\g$ and $r_0$ so that the assumptions and Lemma \ref{lem:weissfacts}(3) imply that $W(y,r) < \frac{\theta+1}{2}|B_1|$ for  $y\in B_{\g r}$. Up to taking $r_0$ smaller and using Theorem \ref{thm:lim}(6), this implies that $W(y,0+)<|B_1|$ for all such $y$. Now take any point $y$ in $\Wb \cap B_{\g r} \cap Z$: we have (just from the fact that $\Wb$ is quasiopen) that
	\[
	\lim_{s\searrow 0} \frac{|B_{s}(y)\cap \Wb|}{|B_s|} = 1.
	\]
	This means that $W(y,0+) \geq |B_1|$, which is a contradiction. Hence $\Wb \cap B_{\g r} \cap Z$ is empty. In particular $\Wb \cap \p \Wb \ss \Wb \cap Z$ is also empty, so $\Wb$ is relatively open. On the other hand, $Z \cap \Wb$ is empty, so $Z = \p \Wb$ on $B_{\g r}$.
	
	The remaining conclusion follows from applying Lemma \ref{lem:visc}.
\end{proof}

\section{Regularity of the Boundary}\label{sec:reg}

In order to discuss further regularity properties of the free boundary $\p \Wb$, or the potentially larger set $Z$ introduced in Theorem \ref{thm:lim}, it will help to identify several pieces of it which exhibit somewhat different behaviors. We will use the following notation:
\[
Z = \p^* \Wb \cup Z_{AC} \cup Z_{N} \cup Z_{C}.
\]
The definitions of these are as follows. The first of these, $\p^* \Wb$ is the reduced boundary of $\Wb$: these are points of Lebesgue density $\frac{1}{2}$ for $\Wb$ whose blow-ups converge to half-spaces, in the sense that
\[
\frac{\Wb - x}{r} \rightarrow \{z:z\cdot \nu > 0\}
\] 
in $L^1_{\text{loc}}$, for some unit vector $\nu$. The set $Z_{AC}$ is defined as
\[
Z_{AC} = \{x\in Z\sm \p^*\Wb: \lim_{r\searrow 0}\frac{|B_r(x)\cap \Wb|}{|B_r|} < 1 \}.
\]
The limit exists at every point by Lemma \ref{lem:weissfacts}(1). Note that by Theorem \ref{thm:lim}, we have that $\lim_{r\searrow 0}\frac{|B_r(x)\cap \Wb|}{|B_r|}\geq c_*$, so this coincides with the essential boundary of $\Wb$ minus the reduced boundary. The set $Z_N$ is the part of $Z$ contained in the interior of $\Wb$. Then $Z_C$ accounts for the remainder of $Z$.

We will show that $\p^*\Wb$ is relatively open in $Z$ and locally given by analytic hypersurfaces. A theorem of Federer guarantees that $Z_{AC}$ is $\cH^{n-1}$-negligible; we will improve this to say that it has a Hausdorff dimension of at most $n-3$, and its blow-ups are given by stationary Alt-Caffarelli cones (hence the name). We do not study $Z_N$ here, although it is somewhat susceptible to free boundary arguments; Remark \ref{rem:smallermin} suggests it may be made up of the mutual nodal set for several eigenfunctions of $\Wb$, in which case it would be contained in a union of analytic manifolds, except for a singular set of Hausdorff dimension $n-2$. See \cite{HL} for a description of such nodal sets. 

As for the remainder of the boundary, $Z_C$, we have very little to say about it except to point out that it is comprised of points at which $\Wb$ has Lebesgue density $1$. We suspect that at every point in this set, $\p \Wb$ may be described as follows: there is a hyperplane $H$ and two $C^1$ graphs $\g_1\leq\g_2$ over $H$ (which coincide at the point $x$) such that locally, $\Wb^c=\overline{\{\g_1<\g_2\}}$ is the cusp-like region between them. Note that if such points exist, depending on the rate of opening of the cusp, the set $\{(x',\g_1(x')): x' \in \p \{z':\g_1(z')<\g_2(z')\}\}$ (this is the set $Z_C$) may not be regular for the Dirichlet problem. This would mean that our choice of representative for $\Wb$ may be genuinely quasiopen, in that it would include some of $Z_C$.

The first step to proving this would be to show the lower bound
\begin{equation}\label{eq:lb}
\max_{k\in\cK, y\in B_r(x)} |u_k(y)|\geq c r
\end{equation}
for every $r$ small enough and every $x\in \p \Wb$. This, however, appears to be outside the scope of our technique, and we leave it as an open question.

\subsection{The Reduced Boundary}

We begin with a lemma that shows that at flat points in $\Wb$ (points where $\p \Wb$ is close to a hyperplane), the eigenfunctions $u_k$ admit a lower bound \eqref{eq:lb}.

\begin{lemma}\label{lem:wklb} Let $\Wb,F$ be as in Theorem \ref{thm:lim}, and $\theta<1$. Then there are constants $c_0,\e,$ and $r_0$ depending only on $F,\Wb,\theta,N$, and $n$, such that provided $0\in \p \Wb$, $W(0,r)<\theta <1$, and there is a unit vector $\nu$ such that
	\begin{equation}\label{eq:wklbh1}
	|\{y\in B_r:y\cdot \nu > 0\}\cap \Wb| \leq \e |B_r|,
	\end{equation}
for an $r<r_0$, then
\[
\max_{k\in \cK,y\in B_r} |u_k(y)|\geq c_0 r.
\]
\end{lemma}

\begin{proof}
	First, select a $\g = \g(\theta)$ such that $\{u_k\}_{k\in \cK}$ constitute a viscosity solution on $B_{\g r}$ with parameters $\{\xi_k\}$ and $C$, using Lemma \ref{lem:viscsol}.
	
	Consider any point $x\in \p \Wb \cap B_{\g r}$. If $x\cdot \nu \geq \eta \g r$, then we have
	\[
	c\eta^n \g^n |B_{r}| \leq |B_{\eta \g r}(x) \cap \Wb|\leq \e |B_r|, 
	\]
	where the lower bound is from Theorem \ref{thm:lim}(4) (using here that $r_0$ is sufficiently small), while the upper bound is by assumption \eqref{eq:wklbh1}. This implies that $\eta \leq C \e^{\frac{1}{n}}$. We may therefore assume that $\eta<\frac{\g}{2}$, and that $B_(\g r/8)(\frac{3 \g r}{4} \nu) \cap \Wb = \emptyset$. If we consider the family of balls
	\[
	B_{\frac{\g r}{8}}(t \g r \nu),
	\]
	there must be a maximal $t\in [\frac{1}{8},\frac{3}{4})$ such that the the ball's intersection with $\p \Wb$ is empty; denote the corresponding ball by $B_*$, the ball with the same center but double the radius by $2B_*$, and take $z\in \p B_*\cap \p \Wb$. 
	
	Assume for contradiction that the conclusion fails, and so
	\[
	\max_{k\in \cK,y\in B_r} |u_k(y)|\geq c_0 r.
	\]
	Now select a function $\z$ which solves the following Dirichlet problem:
	\[
	\begin{cases}
	-\triangle \z =  \l_N(\Wb) &\text{ on } 2B_*\sm B_*\\
	\z = 0 & \text{ on } \p B_*\\
	\z = 1 & \text{ on } \p 2B_*.
	\end{cases}
	\]
	For $r$ small enough, this function is positive on $2B_*\sm B_*$ and has
	\[
	|\n \z | \leq \frac{c(n)}{r}
	\]
	on $\p B_*$. Now using the maximum principle, we have that $|u_k|\leq c_0 r \z$ on $2B_*\sm B_*$, for every $k\in \cK$. Applying the viscosity property (V3) with $\b_k = c_0 r |\n \z|$ gives that
	\[
	\sum_{k\in \cK} \xi_k c_0^2 c^2(n) \geq 1. 
	\]
	Provided $c_0$ is small enough, this gives a contradiction.
\end{proof}

\begin{theorem}\label{thm:reduced} The set $\p^* \Wb$ is relatively open in $Z$, and locally coincides with an analytic hypersurface.	
\end{theorem}

\begin{proof}
	Let us assume that $0\in \p^* \Wb$. Then there is a unit vector $\nu$ such that
	\[
	\frac{|\{y\in B_r:y\cdot \nu \leq 0\}\triangle \Wb|}{|B_r|}\rightarrow 0 
	\]
	as $r\rightarrow 0$. Moreover, from Lemma \ref{lem:weissfacts}(1), we have $W(0,r)\rightarrow \frac{|B_1|}{2}$ as $\r \searrow 0$. Using Lemma \ref{lem:wklb}, we immediately have
	\begin{equation}\label{eq:reducedi1}
	\max_{y\in B_r,k\in \cK} |u_k(y)| \geq c_0 r
	\end{equation}
	for all $r$ small enough, and $\{u_k\}$ constitute a viscosity solution with parameters $\{\xi_k\}$ and $C$ from Lemma \ref{lem:viscsol}.
	
	Consider the blow-up limits (the convergence is local uniform)
	\[
	v_k(x) = \lim_{r\rightarrow 0} \frac{u(r x)}{r}
	\]
	for $k\in \cK$. Arguing as in the proof of Corollary \ref{cor:ld}, we have that the $v_k$ are Lipschitz continuous, and $v_k =0$ on the half-space $\{x\cdot \nu \geq 0\}$. Using the Weiss monotonicity formula of Theorem \ref{thm:lim}(6), the $v_k$ are $1$-homogeneous. Moreover, each of the $v_k$ is harmonic on the set $\{v_k\neq 0\}$. Finally, there must be at least one $k$ for which $v_k$ is not identically $0$, from the lower bound \eqref{eq:reducedi1}, and in fact
	\[
	\max_{B_1, k\in \cK} |v_k| \geq c_0.
	\]
	
	It follows that for each $k\in \cK$, either $v_k$ is identically $0$, or $v_k$ is of the form
	\[
	v_k(x) = \a_k (x\cdot \nu)_-
	\]
	for some constant $\a_k$, with at least one $\a_k$ nonzero (these are simply the only 1-homogeneous harmonic functions supported on a half-space; see \cite{AC}). From this, we see that $\p \Wb/r \rightarrow \{x\cdot \nu =0\}$ in the Hausdorff topology. Applying \ref{lem:visc}, it follows that $\{v_k\}$ and $\{x:x\cdot \nu <0\}$ are a viscosity solution with parameters $\xi_k$; this easily implies that
	\[
	\sum_{k=1}^N \a_k^2 \xi_k = 1.
	\]
	
	From the uniform convergence, for every $\e>0$, there is a $r>0$ such that
	\[
	\max_{B_r,k}|u_k - \a_k (x\cdot \nu)_-|\leq \e r.
	\] 
	Now applying \cite{KL}[Theorem 7.2] we obtain the conclusion. Note that in that theorem, the lower bound hypothesis is not actually needed at every scale, just at the initial scale. 
\end{proof}

\subsection{The Essential Boundary}

\begin{definition} We say that a continuous function $v\geq 0$ is \emph{weakly stationary for the Alt-Caffarelli problem} if
	\begin{enumerate}[({W}1)]
		\item  It is a viscosity solution on $\R^n$ with $N=1$, the set $\W=\{v>0\}$, and parameters $\xi_0=\xi_1 = 1$.
		\item  The function $v$ is harmonic on $\{v>0\}$ (and this set is nonempty).
		\item 	\[
				0< \liminf_{r\searrow 0} \frac{|B_r(x)\cap \{v>0\}|}{|B_r|}\leq  \limsup_{r\searrow 0} \frac{|B_r(x)\cap \{v>0\}|}{|B_r|} <1
				\]
				at every point $x\in \p \{v>0\}$
		\item $v$ is Lipschitz continuous.
		\item the Weiss energy for $v$,
		\[
		W_{AC}(v;x,r) = \frac{1}{r^n}\int_{ B_r(x)} |\n v|^2 + 1_{\{v>0\}} - \frac{1}{r^{n+1}} \int_{ \p B_r(x)} v^2
		\] 
		satisfies
		\begin{equation}\label{eq:wkacd1}
		W'_{AC}(v;x,r) \geq \frac{2}{r^{n+2}} \int_{\p B_r(x)} (v - r v_r)^2;
		\end{equation}
		and
		\begin{equation}
		W_{AC}(v;x,0+)\geq \frac{1}{2}|B_1|
		\end{equation}\label{eq:wkacd2}
		for every $x\in \p \{v>0\}$.
	\end{enumerate}
\end{definition}
It is not difficult to verify that (a-c) imply (d-e), but that will not be relevant below. Let us note, however, the following properties of such $v$:

\begin{proposition}\label{prop:wkac} Let $v$ be weakly stationary for the Alt-Caffarelli problem. Then
	\begin{enumerate}
		\item There is a $c_*$ such that $|B_r(x)\cap \{v>0\}|\geq c_* r^n$ for all $x\in \p \{v>0\}$ and $r>0$.
		\item The set $\{v>0\}$ has locally finite perimeter.
		\item $\p^* \{v>0\}$ is relatively open in $\p \{v>0\}$, and is locally given by analytic hypersurfaces.
		\item If $v$ satisfies just (W2), (W4), and \eqref{eq:wkacd1} in (W5), then
		\[
		\lim_{s\searrow 0} \frac{|B_s(x)\cap \{v>0\}|}{s^n} = W_{AC}(v;x,0+)
		\]
		and
		\[
		|W_{AC}(v;x,r) - W_{AC}(v;y,r)|\leq C \frac{|x-y|}{r}
		\]
		for all $x,y\in \p \{v>0\}$ and $r>0$.
	\end{enumerate}
\end{proposition}

\begin{proof}[Sketch of Proof.]
	For (1), this is immediate from combining (W4) and (W5):
	\[
	\frac{|B_1|}{2} \leq W(x,r) \leq \frac{1}{r^n}\int_{\{v>0\}\cap B_r(x)} |\n v|^2 +1 \leq C\frac{|B_r(x)\cap \{v>0\}|}{r^n}.
	\]
	Then (3) may be checked by following the proofs of Lemma \ref{lem:wklb} and Theorem \ref{thm:reduced}, but with $v$ in place of $u_k$ everywhere. For (2), note that $v$ is subharmonic on $\R^n$, so $\triangle v$ may be represented by a positive Borel measure and
	\[
	\triangle v (B_r(x)) \leq \int_{\p B_r(x)}|\n v| \leq C r^{n-1}
	\]
	on every ball, from (W4). On the other hand, using that $\triangle v \geq 0$ and (3),
	\[
	\triangle v(B_r(x)) \geq \triangle v(B_r(x)\cap \p^* \{v>0\}) = \int_{B_r(x)\cap \p^* \{v>0\}} |\n v| d\cH^{n-1} = \cH^{n-1}(B_r(x)\cap \p^* \{v>0\}).
	\]
	Combining, this gives that $\p^* \{v>0\}$ has locally finite $\cH^{n-1}$ measure. A theorem of H. Federer (\cite{F}[Theorem 4.5.11]) now implies that $\{v>0\}$ has finite perimeter. The final fact (4) may be checked by following the proof of Lemma \ref{lem:weissfacts}.
\end{proof}

\begin{lemma}\label{lem:dimred} Let $v$ be weakly stationary for the Alt-Caffarelli problem, and also $1$-homogeneous. Then $\cH^{n-3+\t}(\p\{v>0\}\sm \p^*\{v>0\}) = 0$ for every $\t>0$.
\end{lemma} 

\begin{proof}
	We first check that if $\p \{v>0\}\sm \p^*\{v>0\}$  is nonempty, then $n\geq 3$. Indeed, if $n=2$, there are only two possibilities for $\{v>0\}$: either $\{v>0\}$ is $\{x_2> 0 \}$ in some coordinates, or it is $\{x_2\neq 0\}$ (no cone other than a half-plane can support a nontrivial, nonnegative harmonic function which is homogeneous of degree $1$). The first has $\p \{v>0\}\sm \p^*\{v>0\}$ empty, while the second contradicts (W3).
	
	Next, we show that given any $v$ as in the hypotheses of the Lemma with $\cH^{s}(\p\{v>0\}\sm \p^*\{v>0\}) > 0$ for $s\geq 1$, there is a $w:\R^{n-1} \rightarrow [0,\8)$ which is weakly stationary for the Alt-Caffarelli problem, homogeneous of degree $1$, and has $\cH^{s-1}(\p\{v>0\}\sm \p^*\{v>0\}) > 0$. Together with the basic observation about cones that if $\p\{v>0\}\sm \p^*\{v>0\}$ contains at least one point besides the origin, then it contains a ray (and hence $\cH^{1}(\p\{v>0\}\sm \p^*\{v>0\}) =\8$), this implies the conclusion.
	
	Let us assume that $x\in  \p\{v>0\}\sm \p^*\{v>0\}$ and $x\neq 0$, and consider the blow-up limit (with respect to local uniform convergence)
	\[
	v_r(y) : = \frac{v(x+r y)}{r}\rightarrow v_0(y).
	\]
	Arguing as in the proof of Lemma \ref{lem:weissfacts}, have that $v_0$ is homogeneous of degree $1$ (from the Weiss formula), Lipschitz continuous, harmonic where positive, has $v_r\rightarrow v_0$ strongly in $H^1$, and has monotone Weiss energy at each point.  This last fact comes from taking the limit
	\[
	 W_{AC}(v_0;y,s) = \lim_{s\searrow 0} W_{AC}(v_r;y,s) = \lim_{s\searrow 0} W_{AC}(v;x+ry,sr),
	\]
	so that
	\begin{align*}
	 W_{AC}(v_0;y,b) -  W_{AC}(v_0;y,a) &= \lim_{s\searrow 0} W_{AC}(v;x+ry,br) - W_{AC}(v;x+ry,ar) \\
	 &\geq \lim_{r\searrow 0} \int_{ar}^{br} \frac{2}{t^{n+2}} \int_{\p B_t(x+ry)} (v - r \p_r v)^2\\
	 & = \lim_{r\searrow 0} \int_{a}^{b} \frac{2}{t^{n+2}} \int_{\p B_t(y)} (v_r - r \p_r v_r)^2\\
	 & = \int_{a}^{b} \frac{2}{t^{n+2}} \int_{\p B_t(y)} (v_0 - r \p_r v_0)^2.
	\end{align*}
	This also gives that
	\begin{equation}\label{eq:dimredi2}
	W_{AC}(v_0;y,s) \geq \frac{|B_1|}{2},
	\end{equation}
	so we may deduce that $v_0$ satisfies (W5) (and, from earlier, (W4)).  
	 
	 We also have 
	\[
	\triangle v_r \rightarrow \triangle v_0
	\]
	as distributions. Now, we have that $v$ is subharmonic on $\R^n$, and
	\[
	\triangle v(E) \geq \triangle v(E \cap \p^* \{v>0\}) \geq \int_{ \p^*\{v>0\}\cap E} |v_\nu|d\cH^{n-1} \geq \cH^{n-1}(\p^*\{v>0\}\cap E)
	\]
	for any Borel set $E$, interpreting the Laplacian as a measure. Take any $y\in \p \{v>0\}$: then from (V3) and the relative isoperimetric inequality (using Proposition \ref{prop:wkac}(2) here; see \cite{F}[4.5.3 and 4.5.6(4)])
	\begin{align}
	(\cH^{n-1}(B_r(y)\cap  \p^*\{v>0\}))^\frac{n}{n-1} &\geq c \min \{ |B_r(y)\cap \{v>0\}|,|B_r(y)\sm  \{v>0\}| \}\nonumber\\
	&  > c(y) r^n>0 \label{eq:dimredi1} 
	\end{align}
	for all $r$ small enough. In particular, this implies $\p^* \{v>0\}$ is dense in $\p \{v>0\}$, so $\p \{v>0\} = \supp \triangle v$. 
	
	Up to passing to a subsequence, we may set $H = \lim \p \{v_r>0\}= \lim \supp \triangle v_r$ in Hausdorff topology. It follows that  
	\[
	\supp \triangle v_0 \ss \lim \supp \triangle v_r = H;
	\]
	this is true of the supports of any sequence of measures converging in weak-* topology.  On the other hand, we also have that $\p\{v_0>0\}\ss \supp \triangle v_0$: indeed, take a point $y\in \p\{v_0>0\}$. If $y\notin \supp \triangle v_0$, then on some small ball $B_r(y)$ $v_0$ is harmonic, so by the strong maximum principle either $v_0=0$ or $v_0>0$ on the ball. Either one violates $y \in \p \{v_0>0\}$, so
	\begin{equation}\label{eq:dimredi3}
	\p \{v_0>0\} \ss \supp \triangle v_0 \ss H.
	\end{equation}
	
	Let us check that $v_0$ is not identically $0$. From \eqref{eq:dimredi1}, we have that
	\[
	\triangle v(B_r(x)) \geq c(x) r^{n-1}. 
	\]
	It follows that
	\[
	\triangle v_0(B_1) \geq c(x) >0,
	\]
	and so $v_0$ cannot be identically $0$. This shows that $v_0$ satisfies (W2).

	Note that at this point we may verify property (W3), using the Weiss formula. The main point is that
	\[
	\lim_{s\rightarrow 0}\frac{|B_s(y)\cap \{v_0>0\}|}{|B_s|} = \frac{1}{|B_1|} \lim_{s\rightarrow 0} W_{AC}(v_0; y,s),
	\]
	from Proposition \ref{prop:wkac}(4). Then the lower bound in (W3) follows immediately from the lower bound \eqref{eq:dimredi2}. For the upper bound, note that $W_{AC}(v_0;0,s)=c_0<1$ for all $s$, and as $R\rightarrow \8$,
	\[
	W_{AC}(v_0;y,R) \leq  W_{AC}(v_0;0,R) + C \frac{|y|}{R} \rightarrow  c_0<|B_1|,
	\]
	with the left inequality from Proposition \ref{prop:wkac}(4). Then the monotonicity implies $W_{AC}(v_0,y,0+)<|B_1|$ as well. One easy consequence of (W3) is that $|\p \{v_0>0\}|=0$.
	
	Our next goal is to check that $\p \{v_0>0\} = H$. Assume that this is not the case, and select a smooth open subset $V_\e$ compactly contained $\text{Int}\{v_0=0\}\cap B_1$ such that
	\[
	|(\{v_0=0\}\cap B_1)\sm V_\e|\leq \e
	\]
	(this is possible as $|\p \{v_0 = 0\}| = 0$). We have that $v_r\rightarrow 0$ uniformly on a neighborhood of $V_\e$. Find a finite collection of balls $B^{k}$ which cover $V_\e$ and $v_r\rightarrow 0$ on their doubles $2B^k$. Then 
	\begin{align*}
	\triangle v_r (B^k) &\leq \fint_{1}^{3/2}\triangle v_r (t B^k) dt \\
	&\leq \fint_{1}^{3/2}\int_{ \p tB^k} |\n v_r|d\cH^{n-1} dt \\
	&\leq C (\int_{\frac{3}{2}B^k} |\n v_r|^2)^{\frac{1}{2}} \\
	&\leq C (\int_{2B^k} | v_r|^2)^{\frac{1}{2}} \rightarrow 0.
	\end{align*}
	The last step used the Caccioppoli inequality on $v_r$. Summing over $k$, we see that
	\[
	\triangle v_r (V_\e)\rightarrow 0. 
	\]
	This gives
	\[
	\cH^{n-1} ( \p \{v_r>0\}\cap  V_\e )\leq \triangle v_r (V_\e)\rightarrow 0
	\]
	as well, from \ref{eq:dimredi1}.
	
	Applying the relative isoperimetric inequality in $V_\e$, the above leads to one of two alternatives: either $|\{v_r>0\}\cap V_\e|\rightarrow 0$, or $|V_\e \sm \{v_r>0\}|\rightarrow 0$. The latter, however, leads to a contradiction for $\e$ small enough: we know from (W3) that
	\[
	\lim \frac{|B_1\cap \{v_r >0\}|}{|B_1|}\leq c <1.
	\]
	We also know that as for each $y\in \{v_0>0\}$, $v_r(y)>0$ for sufficiently small $r$, 
	\[
	\lim |B_1 \cap (\{v_0 > 0\} \sm \{v_r>0\})| = 0. 
	\]
	This means that
	\[
	\lim |B_1 \cap (\{v_0 = 0\} \sm \{ v_r = 0\}) \geq 1-c >0,
	\]
	and so
	\[
	\lim |B_1 \cap (V_\e \sm \{ v_r = 0\}) \geq 1- c - \e.
	\]
	Hence for $\e$ smaller than $1-c$ and $r$ small enough, the second alternative is impossible.
	
	We have shown that 
	\begin{equation}\label{eq:dimredi4}
	|\{v_r>0\}\cap V_\e|\rightarrow 0.
	\end{equation}
	On the other hand, from Proposition \ref{prop:wkac}(1), we have
	\[
	\frac{B_s(y)\cap \{v_r> 0\}}{s^n} \geq c_*
	\]
	uniformly in $y\in \p\{v_r>0\}$, $r$, and $s$. It follows that $H\cap V_\e$ is empty: for any $y_r \in \p\{v_r>0\}\rightarrow y\in V_\e$, there is a small ball $B_s(y_r)\ss V_\e$ with $|B_s(y_r)\cap \{v_r>0\}|>c_* s^n$; for small $r$ this contradicts \eqref{eq:dimredi4}. Repeating for every $\e>0$, we deduce that $H\ss \p \{v_0>0\}$, which together with \eqref{eq:dimredi3} establishes the claim.
	
	It follows from Lemma \ref{lem:visc} that $v_0$ is a viscosity solution, verifying property (W1). We have at this point established that $v_0$ is weakly stationary for the Alt-Caffarelli functional.
	
	Next we will show that $\p_{e} v_0=0$, where $e = \frac{x}{|x|}$. To see this, take any $y\in \{v_0>0\}$; then for $r$ small, we have $v_r(y)>0$ and
	\[
	\p_e v_0(y) = \lim_{r\rightarrow 0} \p_e v_r(y).
	\] 
	Then
	\[
	\p_e v_r(y) = \p_e v(x+ry) = \p_{\frac{x+ry}{|x+ry|}} v(x+ry) + O(1-e \cdot \frac{x+ry}{|x+ry|}) = \frac{v(x+ry)}{|x+ry|} + O(r).
	\]
	The second step used that $v$ is Lipschitz, while the third step used the homogeneity of $v$. Using that $v(x)=0$, the right-hand side is controlled by $r$, so taking the limit gives $\p_e v_0=0$. Therefore $\{v_0>0\}$ is cylindrical, invariant in the $e$ direction.
	
	Finally, we may easily check that the trace $v_*$ of $v_0$ to a hyperplane orthogonal to $e$ is weakly stationary for the Alt-Caffarelli problem. Indeed, (W1-4) are trivial. The monotonicity of the Weiss energy \eqref{eq:wkacd1} may be verified as in \cite{KL}[Section 8], while the lower bound \eqref{eq:wkacd2} can be inferred from using Proposition \ref{prop:wkac}(4) on both $v_0$ and $v_*$:
	\begin{align*}
	\frac{|B_1|}{2} &\leq W_{AC}(v_0;x,0+)\\
	& = \lim_{s\searrow 0+} \frac{|B_s(x)\cap \{v_0>0\}|}{s^n}\\
	& = \lim_{s\searrow 0+} \frac{|B'_s(x')\cap \{v_*>0\}|_{n-1}}{s^{n-1}}\\
	& = W_{AC}(v_*;x',0+).
	\end{align*}
	Here $x'$ and $B'$ represent the projections of $x$ and $B$ onto the $n-1$ dimensional hyperplane orthogonal to $e$.
	
	The size of the singular set $\p\{v_0>0\}\sm \p^* \{v_0>0\}$ can now be estimated using standard arguments from geometric measure theory; see \cite{Giusti}[Chapter 11].
\end{proof}

\begin{theorem}\label{thm:ACset} If $n=2$, then $Z_{AC}$ is empty. If $n\geq 3$, then $\cH^{n-3+\t}(Z_{AC})=0$ for each $\t>0$. 
\end{theorem}

\begin{proof}
	Assume $0\in Z_{AC}$. We will show that a blow-up limit of $Z$, in Hausdorff topology, is the positivity set of a function which is weakly stationary for the Alt-Caffarelli problem and also 1-homogeneous. Combined with Lemma \ref{lem:dimred} and standard measure-theoretic arguments, this implies the conclusion (see \cite{Giusti}). Note that, for small enough $r$, $W(0,r)\leq c < |B_1|$, so $W(x,r)<1$ on a small neighborhood of $0$, using Lemma \ref{lem:weissfacts}(1,3). It follows that the Lebesgue density of $\Wb$ at every point in $Z \cap B_r$ is strictly less than $1$, so $Z\cap B_r$ is contained in $\p^*\Wb \cup Z_{AC}$. In particular, $(Z\sm \p \Wb)\cap B_r$ is empty.
	
	Let us take any sequence $r\searrow 0$ along which
	\[
	u_{k,r}:=\frac{u_k(r \cdot)}{r} \rightarrow v_k
	\]
	locally uniformly for each $k\in \cK$, $\Wb/r \rightarrow V$ as sets of locally finite perimeter, $\p \Wb /r \rightarrow H$ in Hausdorff topology, and
	\[
	0<\lim_{r\searrow 0} \frac{|B_r \cap \Wb|}{|B_r|} <1.
	\]
	Arguing as in the proof of Lemma \ref{lem:weissfacts}(1), we have that each $v_k$ is homogeneous of degree $1$, and the $u_{k,r}\rightarrow v_k$ strongly in $H^1$. Our goal will be to show that not all of the $v_k$ are $0$ everywhere.
	
	Using Theorem \ref{thm:reduced}, we have that for any Borel set $E$,
	\[
	\triangle |u_k|(E\cap \p^* \Wb) = \int_{E\cap \p^* \Wb} |(u_k)_\nu| d\cH^{n-1} ,
	\]
	interpreting $-\triangle |u_k|$ as a measure. Moreover, we have
	\[
	\sum_{k=1}^N \xi_k |(u_k)_\nu|^2 = 1
	\]
	at every point of $\p^*\Wb$, so
	\[
	\sum_{k\in \cK} |(u_k)_\nu| \geq c
	\]
	at every point as well. This gives
	\[
	\sum_{k\in \cK}\triangle |u_k|(E\cap \p^* \Wb) \geq c \cH^{n-1}(\p^*\Wb \cap E).
	\]
	On the other hand, we have that $\triangle |u_k| \geq 0 $ on $\{u_k=0\}$ (indeed, this is true of any continuous function), and also
	\[
	\triangle |u_k| (B_r \cap \{u_k\neq 0\}) \geq - C r^n, 
	\]
	using the fact that $u_k$ solves the eigenfunction equation there. Combining, we have
	\[
	\sum_{k\in \cK}\triangle |u_k|(B_r) \geq c \cH^{n-1}(\p^*\Wb \cap B_r) - C r^n.
	\]
	Using the fact that for each $r$ along our blow-up sequence, $0<c_* \leq \frac{|B_r\cap \Wb|}{|B_r|} \leq 1-c_* <1$, we now invoke the relative isoperimetric inequality to give 
	\[
	\sum_{k\in \cK}\triangle |u_k|(B_r) \geq c r^{n-1} - C r^n.
	\]
	From scaling,
	\[
	\sum_{k\in \cK}\triangle |u_{k,r}|(B_\r) \geq c \r^{n-1} - C r \r^n.
	\]
	Then from convergence in the sense of distributions, we have
	\[
	\sum_{k\in \cK}\triangle |v_k|(B_\r) \geq c \r^{n-1}>0
	\]
	for every $\r$. In particular, at least one of the $v_k$ is not identically $0$ (we call this index $k_*$).
	
	This means that $v_{k_*}$ is homogeneous of degree $1$ and harmonic where nonzero. If $v_{k_*}$ is not nonnegative (up to reversing sign), we would have to have $v_{k_*} = \a (x_n)_+ - \b (x_n)_-$, for some $\a,\b>0$ and some choice of coordinates: indeed, the trace of $v_{k_*}$ on $\p B_1$ is a first Dirichlet eigenfunction on the domains $\{v_{k_*}>0\}\cap \p B_1$ and $\{v_{k_*}<0\}\cap \p B_1$, with the same eigenvalue as a half-sphere (the eigenvalue is determined entirely by the degree of homogeneity). Any domain other than the half-sphere with this value for the first eigenvalue in the sphere must have strictly larger volume (see \cite{AC}), so if $v_{k_*}$ changes sign, the two components of $\{v_{k_*}\neq 0\}\cap \p B_1$ are complementary half-spheres, and $v_{k_*}$ has the form specified. This is a contradiction to the fact that the complement of $\Wb$ has positive Lebesgue density at $0$. We conclude that $v_{k_*}$ must be nonnegative; note also that $\{v_{k_*}>0\}$ is connected, by the same argument.
	
	Considering any other $v_k$, we see that it also vanishes on $H$ (which contains $\p\{v_{k_*}>0\}$, arguing as in the proof of Lemma \ref{lem:dimred}) and solves the same eigenfunction equation on $\p B_1 \cap \{v_k\neq 0\}$. From the uniqueness of the first eigenfunction on this region, we learn that all of the $v_k$ are constant multiples of $v_{k_*}$.
	
	Now arguing as in Lemma \ref{lem:dimred}, we have that $\p \Wb/r$ converges to $H$ in Hausdorff topology, and so $\{v_k\}_{k\in \cK}$ and $\{v_{k_*}>0\}$ form a viscosity solution with parameters $\xi_k$. Using that all of the $v_k$ are constant multiples of one another, this is equivalent to $v_{k_*}$ being a viscosity solution to the Alt-Caffarelli problem (i.e. (W1)).
	
	Properties (W2,4) are immediate, while property (W5) is easily passed to the limit. Property (W3) may then be checked as in the proof of Lemma \ref{lem:dimred}, using (W5) and and the fact that $W(0,0+)<1$. This implies the conclusion.
\end{proof}

\section*{Acknowledgments} DK was supported by the NSF MSPRF fellowship DMS-1502852. FL was supported by the NSF grant DMS-1501000.

\bibliographystyle{plain}
\bibliography{degenerate}

\end{document}